\documentclass[11pt]{amsart}

\usepackage{amsmath}
\usepackage{amsfonts}
\usepackage{amsthm}
\usepackage{comment}
\usepackage{url}
\usepackage[dvipsnames]{xcolor}
\usepackage{enumitem}
\setlist[enumerate]{font=\normalfont}
\usepackage{mathtools}
\usepackage[alphabetic]{amsrefs}
\usepackage{latexsym}
\usepackage{amssymb}
\usepackage{graphicx}        
\usepackage{float}
\usepackage{fdsymbol}
\usepackage{tikz,tikz-3dplot,tikz-cd}
\usepackage[all]{xy}
\usepackage{pgfplots}
\pgfplotsset{compat=1.18}
\usepackage[margin=1cm]{caption}
\usepackage{array} % For fixed height
\usepackage{graphicx} % For adjusting row height
\usepackage{tabularx} % For fixed width
\usepackage{adjustbox} % For proper centering of images
\usepackage{mathrsfs}

\usepackage{geometry}
\usepackage{algorithm}
\usepackage[noend]{algpseudocode}

\usepackage{mathtools} 
\mathtoolsset{showonlyrefs,showmanualtags}
\numberwithin{equation}{section}

\usepackage[colorinlistoftodos]{todonotes}
\usepackage[colorlinks=true, allcolors=MidnightBlue, citecolor=Mahogany]{hyperref}
\usepackage{subcaption}

\definecolor{poly}{RGB}{30, 160, 165} %{rgb}{0, 0.7, 0.7}
\usepackage{eucal}

\theoremstyle{definition}
\newtheorem*{theorem*}{Theorem}
\newtheorem{theorem}{Theorem}[section]

\newtheorem{lemma}[theorem]{Lemma}

\newtheorem{proposition}[theorem]{Proposition}

\theoremstyle{definition}
\newtheorem{definition}[theorem]{Definition}
\newtheorem{example}[theorem]{Example}

\newtheorem*{example*}{Example}

\theoremstyle{remark}
\newtheorem{remark}[theorem]{Remark}
\newtheorem*{remark*}{Remark}

\newcommand\ZZ{\mathbb{Z}}

\newcommand{\RR}{\mathbb{R}}

\newcommand{\Vol}{\operatorname{Vol}}

\renewcommand{\d}{\operatorname{d}}

\newcommand{\dd}{ \, \mathrm{d}}

\DeclareMathOperator{\vol}{vol}
\DeclareMathOperator{\conv}{conv}

\newif\ifkeepslowthings
\keepslowthingstrue  % uncomment this to plot figs
\author{Chiara Meroni}
\author{Jared Miller}
\author{Mauricio Velasco}
\title{Approximation of starshaped sets using polynomials}
\begin{document}

\begin{abstract}
We introduce polystar bodies: compact starshaped sets whose gauge or radial functions are expressible by polynomials, enabling tractable computations, such as that of intersection bodies. We prove that polystar bodies are uniformly dense in starshaped sets and obtain asymptotically optimal approximation guarantees. We develop tools for the construction of polystar approximations and illustrate them via several computational examples, including numerical estimations of largest volume slices and widths.
\end{abstract}

\maketitle
%\tableofcontents

\section{Introduction}
Let $L\subseteq \RR^n$ be a starbody, namely a compact starshaped set containing the origin in its interior. The shape of $L$ is captured by either of two fundamental real-valued functions on the unit sphere $S^{n-1}\subseteq \RR^n$: the {\it gauge (or Minkowski) function} defined as $\gamma_L(x):=\inf\{\lambda\in \mathbb{R}_{>0}:x\in \lambda L\}$ and the  {\it radial function} defined as $\rho_L(x):=\sup\{\lambda\in \mathbb{R}_{>0}: \lambda x\in L\}$. 

Being able to solve optimization problems involving such functions would allow us to approach a wide array of natural questions about starbodies which are not obviously related to $\rho_L$ or $\gamma_L$ but which can be reduced to them. For instance:
\begin{enumerate}
\item {\it What is the largest volume of the intersection of $L$ and an $(n-1)$-plane through the origin?}  The function $\rho_{IL}(x):=\Vol(x^{\perp}\cap L)$ is the radial function of the intersection body $IL$ of $L$ and can be computed by applying the well-known Radon transform to $\rho_L^{n-1}$. The largest volume of hyperplane slices is achieved in the direction $x^*$, where $\rho_{IL}(x^*)$ is maximal.
\item {\it What is the width of $L$?} The width is defined as the largest minimum distance between pairs of parallel hyperplanes enclosing $L$. It can be computed by maximizing the function $w(x):=\gamma_{(\conv L)^\circ}(x)+\gamma_{(\conv L)^\circ}(-x)$ over the unit sphere. Here $(\conv L)^\circ$ is the polar body of the convex hull of $L$.
\end{enumerate}
Despite recent advances~\cite{BdLM:BestSlicePolytope}, there is a rather limited set of tools for approaching the problems above, even in the case when $L$ is a polytope, partly due to the fact that the functions $\gamma_L$ and $\rho_L$ are continuous but typically not differentiable.
\begin{figure}[b]
    \centering
    \includegraphics[width=0.3\linewidth]{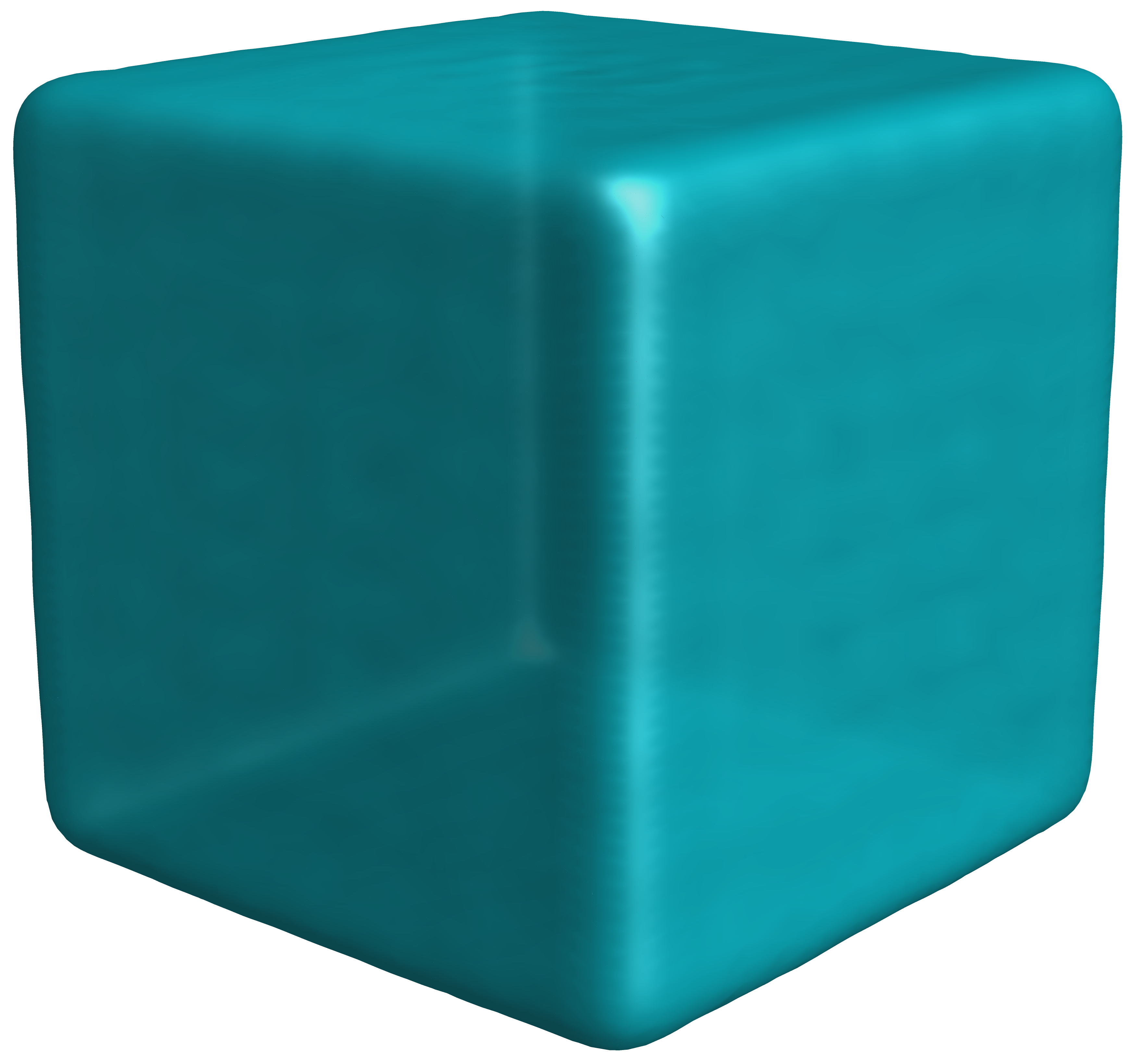}
    \caption{A polystar body approximating the cube.}
    \label{fig:polycube}
\end{figure}
In this article we introduce a new tool, the class of \emph{polystar bodies} (see Figure \ref{fig:polycube}), defined as the set of starbodies whose gauge or radial functions agree with the restriction of a polynomial at all points of the unit sphere $S^{n-1}$ (denoted \emph{polyradial bodies} and \emph{polygauge bodies} respectively). We aim to use polystar bodies to construct uniform approximations of general starbodies. Due to their additional algebraic structure, many invariants are easier to compute for polystar bodies than for general starbodies. If these invariants depend continuously on the body $E$ then the easily computable invariant of a polystar approximation of $E$ will give a good approximation for the value of the true invariant on $E$. Carrying out this strategy in practice requires establishing quantitative uniform approximation guarantees as well as developing effective mechanisms for the construction of polystar approximations and for computations with polystar bodies. Addressing these needs is the central contribution of this article.

More precisely, after introducing polystar bodies and their basic properties in Section~\ref{sec: polystar_basics}, we devote Section~\ref{Sec: polystar_approx} to proving that polystar bodies are capable approximators.  %% this can go later
To state our first result we need an additional piece of notation. Fix $\alpha:=\frac{n-2}{2}$ and for each integer $j$, let $\lambda_j$ be the largest root of the Gegenbauer polynomial $C^{(\alpha)}_j(t)$ (see Section~\ref{sec: Geg_quad} for precise definitions). Denote by $B_R(0) \subset \RR^n$ the standard ball centered at zero with radius $R>0$.
\begin{theorem}\label{thm: starDensity} The following statements hold:
\begin{enumerate}
\item ({\it Polystar density}) The sets of polygauge (resp. polyradial) bodies are dense in the set of starbodies with continuous gauge (resp. radial) function endowed with the supremum norm.
\item\label{item:dist_estimate} ({\it Distance estimates}) Let $L\subseteq \mathbb{R}^n$ be a starbody and $k$ a positive integer.
Let $f \in \{\gamma, \rho\}$. If $f_L$ is Lipschitz continuous with Lipschitz constant $\kappa$, then there exists a polystar body $L'$ with $\deg(f_{L'}) = d = 2k$ such that
\vspace*{-0.1cm}
\[
\|f_L-f_{L'}\|_{\infty}\leq \left(\frac{\pi}{\sqrt{2}}\sqrt{1-\lambda_{k+1}}\right)\kappa .
\]
Furthermore, the right-hand side of the above inequality behaves as 
$\frac{\pi (n-2)}{\sqrt{2}}\, \frac{\kappa}{d}$ as $d\rightarrow \infty$. 
\item({\it Abundance of Lipschitz starbodies}) If $L\subseteq B_R(0)$ is starshaped with respect to every point of the ball of radius $r$ around the origin, then $\kappa:=1/r$ (resp. $\kappa:=R^2/r$) is a valid Lipschitz constant for the gauge (resp. radial) function of $L$. 
\item ({\it Polygauge convexity}) If $L$ is convex, then the polygauge body from part \eqref{item:dist_estimate} is also convex.
\end{enumerate}
\end{theorem}

The proof of Theorem~\ref{thm: starDensity} is constructive. The approximating bodies are obtained by applying some especially tailored polynomial convolutions to the gauge or radial function of the target body following an approach similar to that used in~\cite{FF21:SOSSphere}  and \cite{CriVel24:HarmonicHierarchies} for constructing approximations of nonnegative polynomials via sums-of-squares. The greater generality of gauge/radial functions is mitigated by the fact that Lipschitz continuous functions can be approximated efficiently via polynomial convolutions on spheres thanks to a fundamental result of Newman-Shapiro~\cite{NewSha64:JacksonTheorem}
later generalized by Ragozin~\cite{Ragozin71:PolyApproxSphere},\cite{Ragozin70:PolynomialApproxCompactManifolds}. We use their theorem as our main technical tool and make a modest improvement to their bounds for the optimal approximation constant in Theorem~\ref{thm:NewmannShapiro_operator}. 

Our next result shows that the rate of approximation of $O\left(\frac{\kappa}{d}\right)$ appearing in Theorem \ref{thm: starDensity} part $(2)$ above is asymptotically sharp by providing a Kolmogorov width estimate. To state it precisely, let $C(S^{n-1})$ be the ring of continuous real valued functions in $S^{n-1}$ and for any real number $\kappa>0$, let $\Lambda(\kappa)\subseteq C(S^{n-1})$ be the set of nonnegative functions for which $\kappa$ is a valid Lipschitz constant.

\begin{theorem}\label{thm:Kolmogorov_radial_1}
Let $W\subseteq C(S^{n-1})$ be a finite-dimensional subspace containing the constant functions. If $\beta=\min\left({\rm dim}(W)^{-\frac{1}{n-1}}, \frac{2}{\kappa}\right)$ then the inequality $\sup_{f\in \Lambda(\kappa)} d(f,W)\geq \kappa\frac{\beta}{2}$ holds. In particular, there exists a positive constant $C_0$ such that, for all sufficiently large $d$, there is some starbody $L$ with $f_L \in \{\gamma_L, \rho_L\}$, $f_L \in \Lambda(\kappa)$ for which the inequality
    \[
    \| f_L - p \|_\infty \geq C_0 \frac{\kappa}{d}
    \]
holds for every polynomial $p$ of degree $d$ on $S^{n-1}$.
\end{theorem}

The results discussed so far have two consequences that validate our approach: $(1)$ no other vector space of continuous functions of the same dimension as that of polynomials of degree $\leq d$ can yield better asymptotic approximation rates than the vector space of polynomials on general starbodies and $(2)$ no other mechanism for polynomial approximation behaves better (up to constant multiple) than the polynomial convolutions used for proving Theorem~\ref{thm: starDensity}. Motivated by these observations we next focus on the problem of effectively computing polystar approximations via convolutions, and devote Section~\ref{Sec: polystar_computation} to this topic. We assume that the gauge/radial function $f_L$ of a starbody $L$ is given in the form of a black-box implementation which, given a unit direction $u$, returns the value $f_L(u)$ and wish to use this implementation to construct a good polystar approximation for $L$.

To describe the proposed approach recall that a quadrature rule on the sphere is a pair $(X,W_X)$ where $X\subseteq S^{n-1}$ is a finite set and $W_X:X\rightarrow \mathbb{R}_{>0}$ is a positive function. Such a rule is useful for estimating integrals $\int fd\mu$ via sums $\sum_{z_i\in X}W_X(z_i)f(z_i)$. The rule is exact in degree $m$ if the estimate coincides with the value of the integral for every polynomial $f$ of degree at most $m$. We construct approximate starbodies using quadrature rules in a two-step procedure: Given the gauge/radial function $f_L$ we first estimate its Fourier decomposition via Algorithm 1 below and then apply a special convolution operator via the Funk-Hecke formula obtaining the desired polystar approximation (see Algorithm~\ref{alg:approx_mollified} for details).

\begin{algorithm}[ht!]
    \caption{Compute an approximation of $f_d$, the degree $d$ homogeneous part of $f$}
    \label{alg:approx_fourier}
    \textsc{Input:} $f \in L^2(S^{n-1},\mu)$ with spherical Lipschitz constant $\kappa$, $d>0$, $m>0$.\\
    \textsc{Output:} $\widetilde{f}_d$, so that $\|f_d - \widetilde{f}_d\|_\infty< c_{\kappa,d,n,m}$ for some $c_{\kappa,d,n,m}>0$.
    \begin{algorithmic}[1]
    \State $(X,W_X) \gets $ quadrature rule exact in degrees $d+m$ and $2d$
    \State \Return $\widetilde{f}_d(x) \gets \sum_{z_i \in X} W_X(z_i) f(z_i) Z_d(x,z_i)$
    \end{algorithmic}
\end{algorithm}

Our key technical contribution is Theorem~\ref{cor:algo1} which shows that the existence of a good polynomial approximation combined with a quadrature rule, exact in sufficiently high degree, guarantees that Algorithm~\ref{alg:approx_fourier} yields a provably good uniform approximation for the spherical harmonic components of any Lipschitz function.  In particular the method yields an algorithm for obtaining such approximations using only polynomially many evaluations of the target function, which may be of interest in its own right. Moreover, we show in Theorem~\ref{polySOS} that the functions resulting from this approximation procedure are sums-of-squares in the coordinate ring of the sphere.

Finally, in Section~\ref{Sec: polystar_applications} we showcase some applications of polystar bodies by estimating intersection bodies, largest volume slices, and widths of several interesting starbodies via their polystar approximations. These examples allow us to illustrate the versatility and practical value of the methods introduced in the article.

{\bf Acknowledgements.}
We thank Jes\'us de Loera for his contagious enthusiasm which led to the initiation of this project during ICERM's semester program on Discrete Optimization: Mathematics, Algorithms, and Computation. Chiara Meroni is supported by Dr. Max R\"ossler, the Walter Haefner Foundation, and the ETH Z\"urich Foundation.
Jared Miller was in part funded by Deutsche Forschungsgemeinschaft (DFG, German Research Foundation) under Germany's Excellence Strategy - EXC 2075 – 390740016, and supported by the Stuttgart Center for Simulation Science (SimTech).
Mauricio Velasco is partially supported by Fondo Clemente Estable grant FCE-1-2023-1-176172 (ANII,
Uruguay) and Fondo Clemente Estable grant FCE-1-2023-1-176242.

\section{Starbodies and polystar bodies}
\label{sec: polystar_basics}
In this section we establish the basic language used throughout the article: we introduce the gauge and radial functions of a starshaped set, give a few illustrative examples, and provide a geometric characterization of when such functions are Lipschitz continuous (Proposition~\ref{lem: gaugeLipschitz}). Furthermore, we introduce the key concept of \emph{polystar body}, a starbody whose gauge or radial function is expressible by polynomials, and establish some of its elementary properties (Proposition~\ref{prop: basic_polystar}). 

\subsection{Functions associated to starbodies}
A subset $L\subset \RR^n$ is a \emph{starshaped set with respect to the origin} if for every $x\in L$, the line segment $[0,x]$ connecting $x$ to the origin is entirely contained in $L$.
We will assume $L\subset \RR^n$ to be a \emph{starbody}, namely $L$ is a compact starshaped set which contains the origin in its interior.
Starshaped sets are a generalization of convex sets, which are starshaped with respect to any of their points. For background on starshaped sets we rely on \cite{HHMM:StarshapedSets} and the references therein.
We will work with starbodies via certain associated functions, as explained in the following section.

A starbody $L\subseteq \RR^n$ is completely specified by either of the following two functions: The \emph{gauge} (or \emph{Minkowski}) \emph{function} of $L$ defined as 
\begin{align}
    \gamma_L : S^{n-1} &\to \RR \\
    x &\mapsto \inf\{\lambda>0 \,|\, x\in \lambda L\},
\end{align}
and its reciprocal (multiplicative inverse), is the \emph{radial function} of $L$:
\begin{align}
    \rho_L :  S^{n-1} &\to \RR \\
    x &\mapsto \sup\{\lambda>0 \,|\, \lambda x\in L\}.
\end{align}
The gauge and the radial functions are both strictly positive functions on the sphere. They can also be naturally extended to the whole of $\RR^{n}$ (resp. to $\RR^{n}\setminus\{0\}$) by making $\gamma$ positively homogeneous (resp. $\rho$ homogeneous of degree $(-1)$). We will abuse notation, and use the same symbol for the function on $S^{n-1}$ and for its extension; the domain will be clear from the context. Every positive function on $S^{n-1}$ is the gauge/radial function of some starbody. 

If $L$ is convex, with the origin in its interior, then its gauge/radial function is continuous. This property does not necessarily hold in the nonconvex case, as the following example shows.
\begin{example} \label{ex:starbody_not_continuous}
Let $L\subseteq \RR^2$ be the starbody given by
\begin{equation}\label{eq:starbody_not_continuous}
    L = \{(x,y)\in \RR^2 \,|\, x^2+y^2\leq1\}\cup [(-2,0),(2,0)].
\end{equation}
The gauge (resp. radial) function of $L$ equals $\frac{1}{2}$ (resp. $2$) at $(\pm 1,0)$ and $1$ otherwise. Furthermore the locus of points with respect to which the set $L$ is starshaped consists of the horizontal line segment and in particular has empty interior.
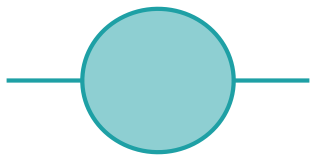
\begin{figure}[ht]
\ifkeepslowthings
    \centering
    \begin{tikzpicture}
\begin{axis}[
    width=3in,
    height=2.5in,
    hide axis,
    xmin=-6,xmax=6,ymin=-6,ymax=4]
\addplot[no markers, samples=100, domain=0:2*pi, variable=\t, style=ultra thick, color = poly, fill = poly!50]
    ({2*cos(\t r)}, {2*sin(\t r)});
\addplot[ultra thick, color=poly, domain=-4:-2] {0};
\addplot[ultra thick, color=poly, domain=2:4] {0};
\end{axis}
\end{tikzpicture}
\fi
    \caption{Starbody with discontinuous gauge and radial function from Example \ref{ex:starbody_not_continuous}.}
    \label{fig:counterex}
\end{figure}
\end{example}

The next proposition shows that the pathologies exhibited in Example \ref{ex:starbody_not_continuous}, lack of continuity of the gauge function and lack of interior of the set of starshaped source-points, are not unrelated.

\begin{proposition}[\cite{BCS23:LipschitzGauge}*{Proposition 2}] \label{lem: gaugeLipschitz} 
Let $L\subseteq \mathbb{R}^n$ be a bounded domain which is starshaped with respect to zero.  The gauge function of $L$ is Lipschitz continuous if and only if $L$ is starshaped with respect to every point of a ball $B_{r}(0)$ centered at zero with radius $r>0$. Moreover, when this condition is satisfied, the function $\gamma_L$ is Lipschitz continuous and $\frac{1}{r}$ is a valid Lipschitz constant for the function $\gamma_L$.
\end{proposition}

Since $L$ is compact, the Lipschitz continuity of the gauge function is equivalent to that of the radial function. Henceforth we will restrict our attention to \emph{Lipschitz starbodies}, defined as starbodies whose gauge (equiv. radial) function is Lispchitz continuous. 

A starbody $L$ is convex if and only if the positively homogeneous gauge function $\gamma: \RR^n\rightarrow \RR$ is convex (or equivalently sublinear).
We will denote by $L^\circ$ the polar convex body, defined as $L^\circ = \{y \in \RR^n \,|\, \langle y, x \rangle \leq 1 \, \forall x \in L\}$, where $\langle \cdot , \cdot \rangle$ is the standard scalar product. If $L$ is convex then the gauge function of $L$ coincides with the \emph{support function} of $L^\circ$, which is by definition
\begin{equation}\label{eq:support_function}
    h_{L^\circ}(x):=\max\{\langle x,y \rangle: y\in L^\circ\}.
\end{equation}
For some special families of convex bodies, radial and support functions are known. For instance, Euclidean balls centered at the origin have constant radial and support function. For polytopes, these functions can be computed from the facet description as follows.

\begin{example}\label{ex:polytope_radialgauge}
    Suppose $P\subseteq \RR^n$ is a polytope containing the origin. If $P$ is given by a facet description of the form $P=\{x\in \RR^n \,|\, \langle H_i,x\rangle \leq 1 \text{ for $i=1,\dots, m$}\}$, then the gauge and the radial functions of the polytope are given respectively by
    \begin{equation}
        \gamma_P(x)=\max_i \, \langle H_i,x\rangle ,
        \qquad
        \rho_P(x)=\min_i\left\{ \frac{1}{\langle H_i,x\rangle}\text{ : $\langle H_i,x\rangle >0$}\right\},
        \qquad
        \text{ for } x \in S^{n-1}\subset\RR^n.
    \end{equation}
This formula for the radial function holds because $\lambda x\in P$ if and only if $\langle H_i, x\rangle =\lambda \langle H_i, x\rangle \leq 1$ for all $i$. The only equations imposing nontrivial constraints on the positive $\lambda$ are those for which  $\langle H_i,\lambda x\rangle >0$, proving the formula. 

For a specific instance, let $P = \conv \{ (-1,-1), (1,-1), (0,1) \}$ be a triangle. Its radial function in spherical coordinates reads
\begin{equation}
    \rho_P(\theta) = 
    \begin{cases}
        -\frac{1}{\sin (\theta)} & -\frac{3 \pi}{4} + 2k\pi \leq \theta\leq -\frac{\pi }{4} + 2k\pi, \\
        \frac{1}{\sin (\theta)+2 \cos (\theta)} & \;\, -\frac{\pi }{4} + 2k\pi \leq \theta\leq \frac{\pi }{2} + 2k\pi, \\
        \frac{1}{\sin (\theta)-2 \cos (\theta)} & \;\, -\frac{\pi }{2} + 2k\pi \leq \theta\leq \frac{5 \pi }{4} + 2k\pi.
    \end{cases}
\end{equation}
\end{example}

Besides completely specifying a starbody $L$, the support, gauge, and radial functions are useful for the effective computation of several of the body's fundamental invariants, for instance its  volume. Indeed, given any starbody $L$, we have 
\begin{equation}
    \vol L = \int_L 1 \, dx = \int_{S^{n-1}}\int_0^{\rho_L(x)} r^{n-1} \, {\rm d}r \, {\rm d} x = \frac{1}{n} \int_{S^{n-1}} {\rho}_L^{\, n}(x) \, {\rm d} x.
\end{equation}

Similarly support, gauge, and radial functions can also be used to define new geometric objects associated to a starbody. This is the case for instance of \emph{intersection bodies}. The intersection body $IL$ of a starbody $L\subset\RR^n$ is the starbody with radial function 
\begin{equation}\label{eq:intersection_body}
\rho_{IL}(x) = \vol (L \cap x^\perp ) = \frac{1}{n-1} \int_{S^{n-1}\cap x^\perp} \rho_L(y)^{n-1} \d\mu(y) = \mathcal{R}\left( \frac{1}{n-1} \rho_L^{n-1} \right)(x),
\end{equation}
where $\mathcal{R}$ denotes the spherical radon transform of any function in $L^2(S^{n-1},\mu)$. 
It is therefore a problem of much interest to compute, or at least to approximate, both gauge and radial functions.

\subsection{Polystar bodies}
When the gauge/radial functions of a starbody are continuous, they allow us to think of a starbody $L\subseteq \RR^n$ as an element of the ring $C(S^{n-1})$ of continuous real-valued functions $f: S^{n-1}\rightarrow \RR$ on the standard unit sphere. This ring is a metric space with the sup-norm $\|\cdot\|_{\infty}$ of uniform convergence and contains a subring $\RR[S^{n-1}]$ consisting of the restrictions of (not necessarily homogeneous) $n$-variate polynomials to the sphere. This ring is covered by the vector spaces $\RR[S^{n-1}]_{\leq d}$ consisting of restrictions of polynomials of degree at most $d$ to the unit sphere as $d$ ranges over the integers.

\begin{definition} A starbody $L$ is a \emph{polygauge} (resp. \emph{polyradial}) body if its gauge function (resp. its radial function) is an element of $\RR[S^{n-1}]$. To ease notation we say that $L$ is a polystar body if it is a starbody which is either polygauge or polyradial.
\end{definition}

\begin{example}\label{ex:polyconvexbody}
Consider the polynomial $p(x,y)=32 x^6 + 32 y + 128$ for $(x,y) \in S^1$. This defines a polygauge body $L_1$ via $\gamma_{L_1}(x,y)=p(x,y)$ and a polyradial body $L_2$ via $\rho_{L_2}(x,y)=p(x,y)$. These are shown in Figure \ref{fig:polyconvexbody}.
\begin{figure}[ht]
\ifkeepslowthings
    \centering
    \begin{tikzpicture}
\begin{axis}[
    width=3in,
    height=2.5in,
    hide axis,
    xmin=-6,xmax=6,ymin=-6,ymax=4]
\addplot[no markers, samples=100, domain=0:2*pi, variable=\t, style=ultra thick, color = poly, fill = poly!50]
                          ({15/(cos(\t r)^6 + sin(\t r) + 4)* cos(\t r)}, {15/(cos(\t r)^6 + sin(\t r) + 4)*sin(\t r)});
\end{axis}
\end{tikzpicture}
\qquad \qquad
\begin{tikzpicture}
\begin{axis}[
    width=3in,
    height=2.5in,
    hide axis,
    xmin=-6,xmax=6,ymin=-4,ymax=6]
\addplot[no markers, samples=100, domain=0:2*pi, variable=\t, style=ultra thick, color = poly, fill = poly!50]
                          ({(cos(\t r)^6 + sin(\t r) + 4)* cos(\t r)}, {(cos(\t r)^6 + sin(\t r) + 4)*sin(\t r)});
\end{axis}
\end{tikzpicture}
\fi
    \caption{The polygauge body $L_1$ and the polyradial body $L_2$ from Example \ref{ex:polyconvexbody}.}
    \label{fig:polyconvexbody}
\end{figure}
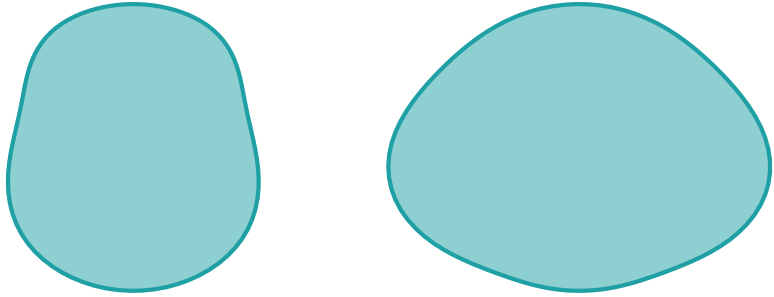
\end{example}
\begin{remark} While the gauge or the radial function of a polystar body $L$ is the restriction of a polynomial in the ambient space to the sphere, in general the homogeneous extension of degree $1$ of $\gamma_L$ or of degree $-1$ of $\rho_L$ to the whole Euclidean space is not a polynomial. Indeed, it is enough to notice that when $L = B_{1}(0)\subset\RR^n$, which satisfies $\gamma_L = \rho_L = 1$ on $S^{n-1}$, the homogeneous extensions of gauge and radial functions at $x\in\RR^n$ are respectively $\|x\|$ and $\frac{1}{\|x\|}$.  
\end{remark}

By the Stone-Weierstrass Theorem the ring $\RR[S^{n-1}]$ is dense in $C(S^{n-1})$ and therefore one should expect that both polyradial and polygauge bodies are a sufficiently rich class to approximate all continuous starbodies \emph{uniformly}.
The main result of the following Section is to provide quantitative measurements for the quality of such approximations as a function of the polynomial degree $d$ and to describe effective mechanisms for their construction when $L$ is a Lipschitz starbody.

We conclude this section with the following proposition, which contains basic properties which distinguish polystar bodies from general starbodies.

\begin{proposition} \label{prop: basic_polystar}
Let $L_1$ be a polygauge starbody, whose gauge function is defined by a polynomial $p\in\RR[S^{n-1}]$. Let $L_2$ be a polyradial starbody, whose radial function is defined by a polynomial $q\in\RR[S^{n-1}]$. Then the following statements hold:
\begin{enumerate}
\item There are no line segments contained in the boundary of $L_1$ or $L_2$;
\item $L_1$ (resp. $L_2$) is convex if and only if the function 
\[
\|x\| \, p\left( \frac{x_1}{\|x\|}, \ldots, \frac{x_n}{\|x\|}\right)
 \quad \left( \text{resp. } \frac{\|x\|}{q\left( \frac{x_1}{\|x\|}, \ldots, \frac{x_n}{\|x\|}\right)} \right)
\]
is convex in $\RR^n$. Furthermore, in either case, the body is strictly convex and the convexity condition can be verified by checking that the Hessian matrix is positive semidefinite at all points of $\RR^n$.
\end{enumerate}
\end{proposition}
\begin{proof} (1) Suppose that there exists a line segment $[a,b]\subset \partial L_1$. Since zero is an interior point of $L_1$ this segment does not contain the origin. Let $\Pi$ be the unique plane containing the origin and the segment $[a,b]$ and let $H\in \Pi$ be a vector which defines an affine line $\langle H,x\rangle =1$ containing the segment $[a,b]$. If $\tilde{\gamma}$ denotes the extension of $\gamma = \gamma_{L_1}$ to $\RR^n$ then our assumption is equivalent to $\tilde{\gamma}(v)=1$ for all $v\in [a,b]$.
By homogeneity this is equivalent to $1=\gamma\left(\frac{v}{\|v\|}\right)\|v\|$ for $v\in [a,b]$. If we use $\theta\in [\theta_0,\theta_1]$ to parametrize the circular arc spanned by the radial projection of $[a,b]$ in $\Pi$ then 
$\gamma(\theta):=\gamma(\cos(\theta),\sin(\theta))=\|v\|^{-1}$ for $v\in [a,b]$. This is equivalent to
\[
\gamma(\theta)= |\langle H,(\cos(\theta),\sin(\theta))\rangle|\text{ for $\theta\in [\theta_0,\theta_1]$}.
\]
Squaring both sides and using the fact that $L_1$ is polyradial we conclude that 
\[
\gamma^2(\theta)=\langle H,(\cos(\theta),\sin(\theta))\rangle^2,
\] 
for theta in a circular arc in $\Pi$. Since this is a Zariski dense subset of the unit circle in $\Pi$ and both sides are regular we conclude that the equality holds at every point of the circle. It follows that $\gamma$ vanishes at some point of the circle since the right-hand side vanishes if we choose a direction perpendicular to $H$. 
Since $\gamma$ and $\rho$ are multiplicative inverses on the sphere, the existence of a line in the boundary of $L_2$ would similarly imply that $\rho(\theta) = \rho_{L_2}(\theta)=|\langle H,(\cos(\theta),\sin(\theta))\rangle|^{-1}$. Since $L_2$ is polyradial this implies that $\rho(\theta)^2$ coincides with the fraction at all points where both are regular. By approaching a zero of the denominator we conclude that the absolute value of $\rho$ grows arbitrarily contradicting the compactness of $L_2$.

(2) The body $L_1$ (resp. $L_2$) is convex if and only if its polar body is convex. This happens if and only if the function on the sphere $p$ (resp. $\frac{1}{q}$) is a support function. Equivalently, if its positively homogeneous extension $\|x\| p\left( \frac{x}{\|x\|}\right)$ (resp. $\frac{\|x\|}{q\left( \frac{x}{\|x\|}\right)}$) is a convex function.\\
A convex body is strictly convex if and only if its boundary contains no lines so the strict convexity follows from part $(1)$. Finally the positively homogeneous gauge function is twice continuously differentiable in $\RR^n$ so its convexity can be characterized by the positive semidefiniteness of its Hessian at every point.
\end{proof}
\noindent
\textbf{Example \ref{ex:polyconvexbody} continued.}
By specializing part $(2)$ of the previous Proposition to the planar case $L_i\subset \RR^2$ with angular coordinate $\theta \in [0,2 \pi]$, one can see that $L_1$ (resp. $L_2$) is convex if and only if 
\[
p(\theta)+p''(\theta) \geq 0
 \quad \left( \text{resp. } \frac{1}{q(\theta)}+\left(\frac{1}{q(\theta)}\right)'' \geq 0 \right).
\]
For instance one can check that the polyradial body of Example~\ref{ex:polyconvexbody} is convex whereas the polygauge body is not. For this, we move to polar coordinates:
$p(x,y)(\theta) = 32 ( (\cos{\theta})^6 + \sin{\theta} + 4 )$, for $\theta\in [0,2\pi]$. We explicitly compute the radius of curvature of both bodies:
\begin{align}
    &L_1 : && -45 \cos (2 x)-90 \cos (4 x)-35 \cos (6 x)+138,\\
    &L_2 : && \frac{1}{32 \left(\sin (\theta )+\cos ^6(\theta )+4\right)^3} \Big( 2 \left(\sin ^2(\theta )+6 \sin (\theta )+8\right)+7 \cos ^{12}(\theta )+2 \cos ^2(\theta ) \\
    & && +42 \sin ^2(\theta ) \cos ^{10}(\theta )+(32-15 \sin (\theta )) \cos ^6(\theta )-30 \sin ^2(\theta ) (\sin (\theta )+4) \cos ^4(\theta ) \Big).
\end{align}
The radius of curvature of $L_1$ is negative at e.g. $\theta=\pi$. The radius of curvature of $L_2$ is a strictly positive function for $\theta\in [0, 2 \pi]$, hence $L_2$ is convex.

\section{Polystar approximation}\label{Sec: polystar_approx}
The main result of this section is the proof of Theorem~\ref{thm: starDensity} which shows that there exist sequences of polystar bodies of degree $\leq d$ capable of approximating any $L$-lipschitz body with a uniform error rate of $O\left(\frac{L}{d}\right)$.  The main tool for constructing the necessary approximating bodies is Theorem~\ref{thm:NewmannShapiro_operator}, a quantitative approximation result for Lipschitz continuous functions on spheres by polynomials. 

If $\mu$ denotes the volume measure in $S^{n-1}$ and $f\in L^2(S^{n-1},\mu)$, then $f$ can be represented, via the partial sums of its harmonic decomposition, as a limit of polynomials in the $L^2$-norm. The question of studying conditions which ensure the harmonic decomposition converges to $f$ in the $L^{\infty}$-norm, a much more useful result for optimization, has a long and illustrious history but typically requires stronger assumptions on $f$ such as high orders of differentiability (or, more precisely, a bound on the modulus of continuity of derivatives of high order as in \cite{Ragozin72:UniformConvergenceSphericalHarmonics}*{Theorem 5},~\cite{Ragozin70:PolynomialApproxCompactManifolds}). 

By contrast, we follow an approach pioneered by Newman and Shapiro \cite{NewSha64:JacksonTheorem} and later generalized by Ragozin \cite{Ragozin71:PolyApproxSphere} which defines a method for re-weighting the Fourier coefficients of the function being approximated to improve the speed of convergence using minimal assumptions on $f$, making it suitable for approximating the typically irregular gauge or radial functions. Using recent results on the roots of Gegenbauer polynomials \cite{DriJor12:boundsZerosOrthogonalPoly}, we improve upon the known rate of convergence of the method, obtaining a more refined constant in Theorem~\ref{thm:NewmannShapiro_operator}.

The tool used for constructing the necessary reweightings are \emph{spherical convolutions with polynomial filters}, a collection of operators whose definition we now recall.
If $u(t)$ is a univariate polynomial which is nonnegative in the interval $[-1,1]$, then the filter $u$ defines a continuous linear \emph{convolution} operator $T_u : L^2(S^{n-1},\mu) \to L^2(S^{n-1},\mu)$ via the formula
\[
T_u(f)(x) = \int_{S^{n-1}} u(\langle x,y\rangle) f(y) \d\mu(y).
\]
In words, the number $T_u(f)(x)$ is the weighted average of the values of the function $f(y)$ on $S^{n-1}$ where the weight is given by the number $u(c)$ on the sphere $S^{n-2}$ defined by $\langle x,y\rangle =c$. In particular, if the filter $u$ is chosen so it concentrates near one then the values $T_u(f)(x)$ should be uniformly close to $f(x)$.  The proof of Theorem~\ref{thm:NewmannShapiro_operator} explains how to explicitly define polynomial filters $u_d$ of degree $d$ with desirable concentration rates as $d$ increases. Making this precise requires a few preliminary notions.

\subsection{Harmonic decompositions and Gegenbauer polynomials}
For $n\geq 2$ let $\Delta:=\sum_{i=1}^n\frac{\partial^2}{\partial x_i^2}$ denote the Laplacian operator in $\RR^n$. A polynomial $f \in \RR[x_1,\ldots,x_n]$ is said to be \emph{harmonic} if $\Delta f(x) = 0$ for all $x\in\RR^n$.
The {\it spherical harmonics of degree $d$} are elements of the subspace $\mathcal{H}_d = \mathcal{H}_d(S^{n-1})\subseteq \RR[S^{n-1}]_{\leq d}$ of dimension $\binom{n+d-2}{d}+\binom{n+d-3}{d-1}$ consisting of the restrictions of homogeneous harmonic polynomials of degree $d$ to the unit sphere. This subspace can also be characterized using the inner product in $L^2(S^{n-1},\mu)$ as
\[
\mathcal{H}_d = \RR[S^{n-1}]_{\leq d}\cap\left(\RR[S^{n-1}]_{\leq d-1}\right)^{\perp}.
\]
It follows that $\RR[S^{n-1}]_{\leq d}=\bigoplus_{j\leq d} \mathcal{H}_j$
and by the Stone-Weierstrass Theorem that the spherical harmonics give an orthogonal decomposition
\[
L^2(S^{n-1},\mu)=\overline{\bigoplus_{d\in \mathbb{N}} \mathcal{H}_{d}}.
\]
In particular, every $f\in L^2(S^{n-1},\mu)$ has a unique expression as a sum $f=\sum_{j=0}^{\infty}f_j$ with $f_j\in \mathcal{H}_j$, the {\it Fourier decomposition} or {\it spherical harmonic decomposition} of $f$. 

As a first application of this decomposition, note that if $u(t)$ is a polynomial of degree $\leq d$ and $f\in L^2(S^{n-1},\mu)$ 
then the convolution $T_u(f)$ annihilates every harmonic component $f_j$ of $f$ with $j>d$. This implies that $T_u(f)$ is always a polynomial, regardless of whether $f$ satisfies any differentiability assumptions.

\subsubsection{Gegenbauer polynomials}\label{sec: Geg_quad}
 If $u(t)$ is an integrable function in $[-1,1]$ then a simple application of the co-area formula shows that
\begin{equation}\label{coarea}
\int_{S^{n-1}} u(\langle x,y\rangle)d\mu(y)= M_n \int_{-1}^1 u(t)(1-t^2)^{\frac{n-3}{2}}dt,
\end{equation}
for some constant $M_n$ depending only on the dimension $n$.
Since the right-hand side is an integral over the interval $[-1,1]$ it is possible to evaluate it exactly on polynomials $u(t)$ via a Gaussian quadrature formula whose basic properties we now review. 

If for each $j\in \mathbb{N}$ we knew polynomials $p_j(t)$ of degree $j$ which form an orthonormal sequence with respect to the inner product $\langle f,g\rangle:=\int_{-1}^1f(t)g(t)\omega(t)dt$, then the integrals of all polynomials $u(t)$ of degree at most $2k-1$ could be computed \emph{exactly} via the \emph{Gaussian quadrature} formula
\[
\int_{-1}^1u(t)\omega(t)dt=\sum_{j=1}^k w_ju(\lambda_j),
\]
where $\lambda_1,\dots,\lambda_k$ are the roots of the polynomial $p_k(t)$, and the weights $w_i$ are defined by integrating the Lagrange interpolators at the same roots, namely
\[
w_i:=\int_{-1}^1 \frac{\prod_{j\neq i}(t-\lambda_j)}{\prod_{j\neq i}(\lambda_i-\lambda_j)}\omega(t)dt.
\]
As a special case, the \emph{Gegenbauer (or ultraspherical) polynomials} form the orthogonal bases required for building such rules with respect to some specific weights.
More precisely, for any real number $\alpha>0$ define the \emph{$(\alpha,j)$-th Gegenbauer polynomial} $C_j^{(\alpha)}(t)$ recursively by the formulas
\begin{equation}\label{eq:Gegenbauer}
\begin{gathered}
    C_0^{(\alpha)}(t)=1\text{ , }C_1^{(\alpha)}(t)=2\alpha t\text{  and }\\
    C_j^{(\alpha)}(t)= \frac{1}{j}\left[2t(j+\alpha -1)C_{j-1}^{(\alpha)}(t)-(j+2\alpha-2)C^{(\alpha)}_{j-2}(t)\right]\text{ if $j\geq 2$}.
\end{gathered}
\end{equation}
For fixed $\alpha$, the Gegenbauer polynomials are orthogonal in $[-1,1]$ with respect to the weight function $\omega_{\alpha}(t):=(1-t^2)^{\alpha-\frac{1}{2}}$ and their norm satisfies the identity
\[
\int_{-1}^{1}\left( C^{(\alpha)}_j(t)\right)^2 \omega_{\alpha}(t)dt= \frac{\pi 2^{1-2\alpha}\Gamma(j+2\alpha)}{j!(j+\alpha)\Gamma(\alpha)^2}.
\]

\subsection{Polynomial approximation of Lipschitz functions via convolutions}\label{sec:best_approx} We are now ready to prove the main result of this Section. We emphasize that Theorem \ref{thm:NewmannShapiro_operator} represents a minor refinement of \cite{NewSha64:JacksonTheorem}*{Theorem 2}, providing a slightly smaller constant in the asymptotic behavior. The proof is added for the reader's benefit, since it highlights the chosen polynomial $u$, and explains why it is an efficient candidate.

For given functions $f(x)$, $g(x)$, we use the asymptotic notation $f\sim g$ as $x\to\infty$ as shorthand for $\limsup_{x\rightarrow \infty} \left|{f(x)}/{g(x)}\right|$ being finite.
Recall the linear operator $T_u : L^2(S^{n-1},\mu) \to L^2(S^{n-1},\mu)$ given by
\[
T_u(f)(x) = \int_{S^{n-1}} u(\langle x,y\rangle) f(y) \d\mu(y).
\]
Fix $\alpha=\frac{n-2}{2}$ and for each integer $j$ let $\lambda_j$ be the largest root of the Gegenbauer polynomial $C^{(\alpha)}_j(t)$. These numbers are fundamental for estimating the quality of our approximations.

\begin{theorem}\label{thm:NewmannShapiro_operator}
    Let $f$ be a Lipschitz function with Lipschitz constant $\kappa$ on $S^{n-1}$. If $d=2k$ then there exists a univariate polynomial $u(t)$ of degree $d$ which is nonnegative in $[-1,1]$ such that  \[
    \| f - T_u(f) \|_{\infty} \leq \left(\frac{\pi}{\sqrt{2}}\sqrt{1-\lambda_{k+1}}\right)\kappa.\]
Furthermore the quantity in the right hand side behaves as $\frac{\pi (n-2)}{\sqrt{2}}\, \frac{\kappa}{d}$ as $d\rightarrow \infty$.
\end{theorem}
\begin{proof}
    Let $d=2k$, $\alpha = \frac{n-2}{2}$, and consider the Gegenbauer polynomial $C^{(\alpha)}_j$ of degree $j$. Denote by $\lambda_j$ the largest root of $C^{(\alpha)}_j$ and define the polynomial
    \begin{equation}\label{eq:u_k_optimal}
    u_{2k}(t) = c \left( \frac{C^{(\alpha)}_{k+1}(t)}{t-\lambda_{k+1}} \right)^2,
    \end{equation}
    where $c$ is a normalization constant chosen to guarantee that $\int_{S^{n-1}} u_{2k}(\langle x,y\rangle) \d\mu(y)=1$.
    Since $f$ has Lipschitz constant $\kappa$, we know that $|f(x)-f(y)|\leq \kappa d(x,y)$ for any pair of points $x,y \in S^{n-1}$. Applying the operator $T_{u_{2k}}$ to this inequality we get that for every $x\in S^{n-1}$ \begin{equation}\label{eq:estimateNS}
    |f(x) - T_{u_{2k}}(f)(x)| \leq \kappa \, T_{u_{2k}}(d(x,y))(x) \leq \kappa \, \sqrt{T_{u_{2k}}(d(x,y)^2)(x)},
    \end{equation}
    where we the last inequality follows from Cauchy-Schwarz in $L^2(S^{n-1},\mu)$. We can estimate the squared geodesic distance on the sphere with a linear polynomial, as in \cite{NewSha64:JacksonTheorem}:
    % quadratic polynomial, instead of a linear approximation used in \cite{NewSha64:JacksonTheorem}:
    \[
    d(x,y)^2 = \arccos^2{(\langle x,y\rangle)} \leq \frac{\pi^2}{2} (1-\langle x,y\rangle),
    % \arccos^2{(\langle x,y\rangle)} \leq \left(\frac{\pi^2}{4}-1\right) (1-\langle x,y\rangle)^2 + 2(1-\langle x,y\rangle).
    \]
and therefore by the nonnegativity of $u_{2k}$,
    \begin{align*}
        T_{u_{2k}}(d(x,y)^2)(x) &\leq \frac{\pi^2}{2} (1- T_{u_{2k}}(\langle x,y\rangle)).
    \end{align*}
By formula~\eqref{coarea}, this implies that for $x,y\in S^{n-1}$
    \[
    T_{u_{2k}}(\langle x,y\rangle)(x) = \frac{\int_S u_{2k}(\langle x,z\rangle) \langle x,z\rangle \d \mu(z)}{\int_S u_{2k}(\langle x,z\rangle) \d \mu(z)}=\frac{\int_{-1}^1 t\, u_{2k}(t) (1-t^2)^{\alpha-\frac{1}{2}} \d t}{\int_{-1}^1 u_{2k}(t) (1-t^2)^{\alpha-\frac{1}{2}} \d t} = \lambda_{k+1}.
    \]
The last equality follows from the orthogonality in $[-1,1]$ of Gegenbauer polynomials with the weight function $(1-t^2)^{\alpha-1/2}$. Indeed, the roots of such polynomial of degree $k+1$ are the nodes of a Gaussian quadrature formula which gives the exact values of integrals for polynomials of degree at most $2(k+1)-1$ for this weight function. In particular, we can use it for $tu_{2k}(t)$. Substitution into~\ref{eq:estimateNS} now yields the claimed inequality.

In the more recent paper \cite{DriJor12:boundsZerosOrthogonalPoly}*{Section 2.3} the authors found the bound 
    \[
        1 - \lambda_{k+1} < \frac{(2\alpha + 1)(2 \alpha + 5)}{4k(k+2\alpha+2) + (2\alpha + 1)(2 \alpha + 5)} \sim \frac{4\alpha^2}{4 k^2} \sim \frac{(n-2)^2}{d^2}.
    \]
    Putting all these ingredients together we get 
    \[
        T_{u_{2k}}(d(x,y)^2)(x) \leq \frac{\pi^2}{2} (1- T_{u_{2k}}(\langle x,y\rangle)) = \frac{\pi^2}{2} (1- \lambda_{k+1})  \quad \hbox{and} \quad \frac{\pi^2}{2} (1- \lambda_{k+1}) \sim \frac{\pi^2}{2} \frac{(n-2)^2}{d^2}.
    \]
    Going back to \eqref{eq:estimateNS} and taking the square root, we obtain the claimed asymptotic behaviour.
\end{proof}

\begin{remark}
    Several earlier works have highlighted the \emph{extremality} of the polynomials $T_u(f)$ for $u$ as in \eqref{eq:u_k_optimal}. After Newmann, Shapiro, \cite{NewSha64:JacksonTheorem} and Ragozin \cite{Ragozin71:PolyApproxSphere}, 
    in the work by Fang and Fawzi \cite{FF21:SOSSphere}, the authors show that the sum-of-squares hierarchy to approximate the maximum of a polynomial on the sphere converges quadratically, at speed $(\frac{n}{d})^2$ where $n-1$ is the dimension of the sphere and $d$ the degree of the hierarchy. This is achieved \cite{FF21:SOSSphere}*{Proposition 7} by bounding the largest root of the Gegenbauer polynomials, and then the error of the inverse operator $T_u^{-1}(f)$.
\end{remark}

\subsection{Constructing polystar approximations}
In this section we prove Theorem~\ref{thm: starDensity} which shows that polyradial and polygauge bodies are dense among all Lipschitz starbodies and provides us with distance estimates. It further gives us an abundance of examples of Lipschitz starbodies and proves that {\it convex polygauge bodies} are dense among all Lipschitz convex bodies. These properties make polystar bodies suitable approximators.

\begin{proof}[Proof of Theorem~\ref{thm: starDensity}]
$(1)$ Since $L$ is a starbody and by definition the origin lies in its interior, the gamma function $\gamma_L$ is strictly positive on $S^{n-1}$ and $\rho_L=\gamma_L^{-1}$, therefore, either both functions are continuous or neither are. By assumption we are in the first case, and the general Stone-Weierstrass Theorem implies that there exist polynomials which, restricted to the sphere, approximate $\gamma_L$ (resp. $\rho_L$) uniformly. This holds because the sphere is compact and polynomials restricted to the sphere are an algebra which separates points and contains the constants. 

$(2)$ Let $f \in \{\gamma, \rho\}$ be the gauge/radial function of $L$. By Theorem~\ref{thm:NewmannShapiro_operator} with the univariate polynomial $u_d(t)$ of degree $d$ chosen optimally as in \eqref{eq:u_k_optimal}, we can construct a polynomial $T_{u_d}(f_L)$ of degree at most $d$ which approximates $f_L$ uniformly, having distance in $O\left(\frac{\kappa}{d}\right)$. Defining $L'$ to be the polystar body with gauge/radial function $f_{L'} = T_{u_d}(f_L)$, the claim holds.

$(3)$ To obtain explicit quantitative bounds for the Lipschitz constant, suppose $L$ is a starbody which is starshaped around every point of the ball $B_{r}(0)$ with radius $r>0$. By Proposition~\ref{lem: gaugeLipschitz}, the gauge function of $L$ is Lipschitz continuous and $\frac{1}{r}$ is an admissible Lipschitz constant. 
Furthermore, 
\[
|\rho_{L}(y)-\rho_{L}(x)| =\frac{|\gamma_{L}(y)-\gamma_{L}(x)|}{|\gamma_L(x)\gamma_L(y)|}\leq \frac{|\gamma_{L}(y)-\gamma_{L}(x)|}{sup_{z\in S^{n-1}}|\gamma_L(z)|^2} = \left(\inf_{z\in S^{n-1}}|\rho_L(z)|^2\right)|\gamma_L(x)-\gamma_L(y)|,
\]
so letting $m:=\inf_{z\in S^{n-1}}|\rho_L(z)|^2$, also $\rho_L$ is Lipschitz continuous, with constant $\frac{m}{r}$.

$(4)$ Finally, we show that if $L$ is convex then the approximating polygauge body $L'$ we have constructed is automatically convex. To see this, recall that the gauge function of $L'$ was built by applying the convolution operator $T_{u_d}$ to the gauge function $\gamma_L$ of $L$:
\[
\gamma_{L'}(x):=T_{u_d}(\gamma_L)=\int_{S^{n-1}} u_d(\langle x,y\rangle)\gamma_L(y)d\mu(y).
\]
A deep result of Kiderlen~\cite{Kiderlen06:EndomorphismsCB}*{Theorem 1.4} shows that for $u_d(t)$ nonnegative on $[-1,1]$ the operator $T_d$ preserves the property of {\it being a support function} (cf. \eqref{eq:support_function}). In other words, the natural extension $\|x\|\gamma_{L'}\left(\frac{x}{\|x\|}\right)$ is convex whenever the natural extension $\|x\|\gamma_{L}\left(\frac{x}{\|x\|}\right)$ of the original function $\gamma_L$ is convex, i.e., whenever $L$ is convex.
For the reader's benefit, we outline the ingredients of Kiderlen's argument. Let $\mathcal{R}_\alpha(\gamma_L(x))$ be the generalized spherical Radon transform, namely the uniform average of the $\gamma_L$-values on the $(n-2)$-dimensional sphere $S\cap \left(\alpha x + x^{\perp}\right)$, and consider the (positive) measure $\tilde{\mu}$ defined on $[-1,1]$ by the univariate polynomial $u_d(t)$. Then, the gauge function of $L'$ can be rewritten as
\[
\gamma_{L'}(x)= \int_{-1}^1 \mathcal{R}_{\alpha}(\gamma_L(x))d\tilde{\mu}(\alpha).
\]
Kiderlen shows that if $L$ is convex then, for every choice of $\alpha\in [-1,1]$, the natural extensions of the functions $\mathcal{R}_{\alpha}(\gamma_L(x))$ are support functions of certain {\it convex} bodies. As a result, the same property holds for any of their weighted averages. We conclude that the polygauge body $L'$, defined as the starbody with gauge function $\gamma_{L'}$, is a convex set as claimed.
\end{proof}

\begin{remark} We will show later (Proposition~\ref{polySOS}) that the gauge/radial function of the polygauge/polyradial bodies constructed in Theorem~\ref{thm: starDensity} can be chosen to be sums of squares in $\RR[S^{n-1}]$.
\end{remark}

\subsection{Asymptotic optimality of polystar approximations}\label{subsec:Kolmogorov}

In this section we show that the approximation error in Theorem~\ref{thm: starDensity} for the radial function of a starbody via the convolution operator $T_u$ where $u$ is a univariate polynomial is, asymptotically, the best possible. We prove this by computing lower bounds on the \emph{Kolmogorov width} of the set of positive functions with a fixed Lipschitz constant. The ideas for the proof are based on results of G.G. Lorentz \cite{Lor60:BoundsDegreeApprox}*{Theorem 1} who uses them to prove lower bounds on the Kolmogorov width of vector spaces of functions with a given modulus of continuity. 

\begin{definition}
    Let $A\subset C(S^{n-1})$ be a set of continuous, real-valued functions on the sphere, endowed with the supremum norm $\| \cdot \|_{\infty}$. Given $N\in \ZZ$, the \emph{Kolmogorov $N$-width} of $A$ is 
    \[
    \mathcal{W}_N(A) = \inf_{\substack{W\subset C(S^{n-1})\\ \dim W \leq N}} \sup_{a \in A} \d(a,W),
    \]
    where the infimum runs over all linear subspaces $W$ of $C(S^{n-1})$ of dimension at most $N$, and the distance function is $\d(a,W) = \inf_{w\in W} \| w-a \|_{\infty}$.
\end{definition}

The Kolmogorov $N$-width measures the smallest (i.e., the best possible) worst-case uniform approximation error among all subspaces $W$ of dimension $N$ in the case when we want to approximate the functions in $A$ uniformly. Therefore, a lower bound of $\delta$ on the Kolmogorov $N$-width implies that for any subspace $W$ spanned by $N$ functions, there is a function $a\in A$ that cannot be approximated by the functions in $W$ with an error smaller than $\delta$.

We start with a preliminary lemma that constructs a family of nonnegative Lipschitz functions.

\begin{lemma}\label{lem:starshaped_add_remove_cones_new}
    Let $\kappa>0$ be a given real number and suppose $v_1,\ldots, v_N \in S^{n-1}$ are a set of points at pairwise (spherical geodesic) distance at least $2\delta$ for some integer $N$. If $\delta<1/\kappa$, then for any choice of signs $\tau\in\{-1,1\}^N$ there exists a function $f_\tau\in C(S)$ satisfying the following properties:
\begin{enumerate}
    \item The function $f_{\tau}$ is strictly positive and Lipschitz continuous with Lipschitz constant $\kappa$;
    \item for $i=1,\ldots,N$ we have $f_\tau(v_i)=1+\tau_i \kappa\delta$ and $f_{\tau} (x) = 1$ for all $x$ with distance at least $\delta$ from all points $v_i$.
\end{enumerate}
\end{lemma}
\begin{proof} For each $\tau\in \{-1,1\}^N$ define
\[f_{\tau}(x):=\begin{cases}
1-\tau_i\kappa (d(x,v_i)-\delta)\text{ , if $d(x,v_i)\leq \delta$}\\
1\text{ , otherwise.}
    \end{cases}\]
where $d(x,z)$ denotes the geodesic distance between points $x,z\in S^{n-1}$. The function is well defined and continuous since the discs $D_i$ of radius $\delta$ around the points $v_i$ are disjoint by our pairwise distance assumption. For any two points $z_1,z_2$ in the interior of the disc $D_i$ we have
\[|f_\tau(z_1)-f_\tau(z_2)|=\kappa|d(v_i,z_1)-d(v_i,z_2)|\leq \kappa d(z_1,z_2)\]
where the last equality used the reverse triangle inequality. Given any two points $z_1,z_2\in S^{n-1}$ we can split the short geodesic joining them into finitely many geodesic paths, each contained in the interior of a single disc or on the outside of all of them. Since the function is constant outside the union of the $D_i$'s such paths contribute zero to the variation of the function along the path and the previous paragraph implies that $\kappa$ is a valid Lipschitz constant for $f_\tau$. Our assumption that $\delta<\frac{1}{\kappa}$ guarantees that $f_\tau(v_i)=1+\tau_i\delta\kappa$ be positive for every $\tau$, ensuring that the $f_\tau$ are positive functions. Finally it is immediate from the formula that $f_\tau$ satisfies the properties in item $(2)$.\end{proof}

For a real $\kappa>0$, let $\Lambda(\kappa)\subseteq C(S^{n-1})$ be the set of positive functions with Lipschitz constant at most $\kappa$. We estimate the Kolmogorov width of $\Lambda(\kappa)$, proving the main result of the Section.

\begin{proof}[Proof of Theorem~\ref{thm:Kolmogorov_radial_1}]
    By a standard volume argument (see, for instance \cite{NewSha64:JacksonTheorem}*{Lemma 2}), the inequality $N\beta^{n-1} <1$ guarantees the existence of $N+1$ points on $S^{n-1}$ at pairwise distance at least $\beta$. Denote these points by $v_1,\ldots,v_{N+1}$ and use Lemma~\ref{lem:starshaped_add_remove_cones_new} with $\delta=\beta/2$ to construct the $2^{N+1}$ functions $f_\tau$ one for each choice of sign $\tau \in \{-1,1\}^{N+1}$. We will use these functions to bound the Kolmogorov width of $\Lambda(\kappa)$.

    Let $W\subset C(S^{n-1})$ be any subspace containing the constant functions, and let $\{g_1,\ldots,g_N\}$ be any basis for $W$. The $N\times (N+1)$ matrix of evaluations of the basis elements $g_j$ at the points $v_i$ has more columns than rows and therefore its kernel is nontrivial. Thus there exist constants $c_1,\ldots,c_{N+1}$ such that $\sum_{i=1}^{N} c_i g_j(v_i) = 0$ for every $j$. Without loss of generality we can rescale the $c_i$ and assume that $\sum_i |c_i| = 1$. 
    
    For every choice of signs $\tau$ there is an optimal choice for the approximation coefficients $a_j$ so that the following error is minimized
    \[
    d(f_\tau, W)=\left\| f_{\tau} - \sum_{j=1}^N a_j g_j \right\|_\infty = \sup_{x\in S^{n-1}} \left|f_\tau(x) - \sum_{j=1}^N a_j g_j (x) \right|.
    \]
Since $W$ contains the constant functions there is another choice of optimal coefficients $\hat{a}_j$ such that
    \[\sup_{x\in S^{n-1}} \left|f_\tau(x) - \sum_{j=1}^N a_j g_j (x) \right|= \sup_{x\in S^{n-1}} \left| (f_\tau(x) - 1) - \sum_{j=1}^N \hat{a}_j g_j (x) \right|.
    \]
In particular if $\nu = \sum_{i=1}^{N+1} w_i \delta_{v_i}$ is any signed measure on the sphere with $\sum |w_i|=1$, then the norm is bounded below by 
    \[
    \left\| f_\tau - \sum_{j=1}^N a_j g_j \right\|_\infty \geq \int_{S^{n-1}} \Big( (f_\tau(x) - 1) - \sum_{j=1}^N \hat{a}_j g_j (x) \Big) \dd \nu(x).
    \]
Using the special measure $\nu = \sum_{i=1}^{N+1} c_i \delta_{v_i}$ constructed from the kernel element we get
    \[
    \| f_\tau - \sum_{j=1}^N a_j g_j \|_\infty \geq \sum_{i=1}^{N+1} c_i (f_\tau(v_i)-1).
    \]
    Consider now the function $f_{\tau^*}$ obtained by the choice $\tau^*_i = \operatorname{sign}(c_i)$, then
    \[
    \| \rho_{L^*} - \sum_{j=1}^N a_j g_j \|_\infty \geq \sum_{i=1}^{N+1} |c_i| |(f_\tau(v_i)-1)| = \kappa \frac{\beta}{2},
    \]
    proving that $\sup_{f\in \Lambda(\kappa)}d(f,W)\geq \kappa \frac{\beta}{2}$ as claimed.
    For the final part of the claim, note that when $W:=\RR[S^{n-1}]_{\leq d}$ is the space of polynomials of degree at most $d$ restricted to the sphere then its dimension is 
    \[
    N = \binom{n+d}{d} - \binom{n+d-2}{d-2}\sim \frac{2}{(n-1)!} d^{n-1}
    \]
    and $\beta<N^{-\frac{1}{n+1}}$. 
So $N^{-\frac{1}{n-1}}$ is bounded below by some constant $\alpha$ times $\frac{n}{d}$ asymptotically in $d$, hence $\sup_{f\in \Lambda(\kappa)}d(f,W)$ is bounded below by $\frac{\alpha n}{2}\frac{\kappa}{d}$ for all sufficiently large $d$ as claimed.
\end{proof}

\section{Computation of polystar approximations via quadratures}\label{Sec: polystar_computation}

So far we have shown that polystar bodies are capable approximators and have provided convolutional formulas for their construction. In this Section we focus on the problem of how to use these formulas in practice. More precisely, we assume that the gauge/radial function $f_L$ of a starbody $L$ is given to us in the form of a black-box implementation which, given a point $x$ on the sphere, returns the value $f_L(x)$, and we wish to use this implementation to construct a good approximating polystar body.
The main result of the Section is Algorithm~\ref{alg:approx_mollified}  which constructs polystar approximations. The key components of the algorithm are a quadrature-based procedure for stably computing uniform approximations of the harmonic components of a function from samples (Theorem~\ref{thm:uniform_harmonic}) together with the remarkable Funk-Hecke formula. Describing these algorithms requires a few preliminaries discussed in the following section.

We remark that the discrete counterpart of our problem, namely the Fourier decomposition of a discrete function on the sphere and the development of efficient algorithms for that, is a well studied problem originated in \cite{DriHea94:ComputingFourier} and further developed in works such as \cites{PST98:FastStableFourier,KunPot03:FastFourier}. For a thorough overview of the history and the results in this direction, see \cite{OSRT20:FastAlgoOrthogonal}. 

\subsection{A spherical quadrature, zonal harmonics and the Funk-Hecke formula}

A quadrature rule on a compact set $M\subseteq \RR^n$ with a measure $\mu$ is a pair $(X, W_X)$, where $X \subset M$ is the finite set of nodes and $W_X:X\to \RR_+$ is a function defining the set of weights of the rule. The quadrature rule provides an estimate of $\int_{M} f(x) d\mu(x)$ as $\sum_{x\in X} W_X(x) f(x)$. We say that the quadrature is exact in degree $d$ if equality holds for every polynomial $f(x)$ of degree at most $d$.

\begin{example}It is shown in~\cite{CriVel24:HarmonicHierarchies}*{Construction 2.1}, that the following procedure yields an explicit quadrature rule for the uniform measure $\mu$ on the sphere $S^{n-1}\subset \RR^n$ for $n\geq 2$, which is exact on multivariate polynomials of degree $d=2k$. We construct the rule inductively on the dimension of the sphere as follows:
\begin{enumerate}
    \item Fix a quadrature rule on $S^1$ having $2(k+1)$ equidistant nodes and equal weights $\frac{\pi}{k+1}$.
    \item For $3\leq \ell \leq n$, suppose that $(Y,W_Y)$ is a quadrature rule on $S^{\ell-2}$, exact in degree $2k$ and let $(Z,W_Z)$ be a Gaussian quadrature rule (see Section~\ref{sec: Geg_quad}) for the weight function $\omega_{\frac{\ell-2}{2}}(t) = (1-t^2)^{\frac{\ell-3}{2}}$ on the interval $[-1,1]$, exact in degree $2k$. Define the rule $(X,W_X)$ on $S^{\ell-1}$ via 
    \begin{gather*}
        X=\left\{\left(z, \sqrt{1-z^2} \, y\right): (z,y)\in Z\times Y\right\}, \\
        W_X\left(z, \sqrt{1-z^2} \, y\right) :=W_Z(z)W_Y(y).
    \end{gather*}
\end{enumerate}
The resulting rule in $S^{n-1}$ has $2(k+1)^{n-1}$ nodes and is exact in degree $2k$ so it can be evaluated using a number of function evaluations which is of polynomial size in the degree. The explicit nature of these rules will allow us to transform error estimates on the nodes and weights of the Gaussian quadrature rules in $[-1,1]$ into error estimates for the nodes and weights on the sphere. 
\end{example}

As discussed in Section~\ref{sec: Geg_quad} it is well known that for any weight function $\omega(t)$ in $[-1,1]$ there exists a Gaussian quadrature rule with $N$ nodes which is exact on polynomials of degree at most $2N-1$. 
The effective computation of nodes $t_i$ and weights $w_i$ for such quadrature rules from a collection of orthogonal polynomials is often done via the Golub-Welsh algorithm \cite{GolWel69:QuadratureRules}*{Section 2}, which we briefly review. 
Given the weight $\omega(t)$, any sequence of polynomials $\{p_j(t)\}_{j=1}^N$ orthogonal with respect to $\langle f,g\rangle:=\int_1^1 f(t)g(t)\omega(t)dt$ satisfies a three-term recurrence relation of the form
\begin{gather*}
    p_{-1}(t) = 0, \quad p_0(t) = 1, \\
    p_j(t) = (a_j t + b_j) p_{j-1}(t) - c_j p_{j-2}(t),
\end{gather*}
for some $a_j, b_j, c_j \in \RR$ with $a_j, c_j >0$. We can write this in matrix notation as 
\[
t \, \mathfrak{p}(t) = T\, \mathfrak{p}(t) + \frac{1}{a_N} p_N(t) e_N,
\]
where $\mathfrak{p}(t)$ denotes the vector of polynomials $(p_0(t), \ldots ,p_N(t))$, $e_i$ is the $i$-th canonical basis vector, and $T$ is the \emph{tridiagonal} matrix given by 
\[
T = \begin{pmatrix}
-\frac{b_1}{a_1} & \frac{1}{a_1} & 0 & 0  & \dots & 0\\[1ex]
\frac{c_2}{a_2} & -\frac{b_2}{a_2} & \frac{1}{a_2} & 0  & \dots & 0\\[1ex]
0 & \frac{c_3}{a_3} & -\frac{b_3}{a_3} & \frac{1}{a_3} & \dots & 0\\[1ex]
\vdots & \vdots & \ddots & \ddots & \ddots & \vdots\\[1ex]
0 & \ldots & 0 & \frac{c_{N-1}}{a_{N-1}} & -\frac{b_{N-1}}{a_{N-1}} & \frac{1}{a_{N-1}} \\[1ex]
0 & \ldots & \ldots & 0 & \frac{c_{N}}{a_{N}} &-\frac{b_N}{a_N} \\
\end{pmatrix}.
\]
As a consequence, if $t_i$ is a zero of $p_N$ then $t_i \mathfrak{p}(t_i) = T \mathfrak{p}(t_i)$, so that $t_i$ is an eigenvalue of $T$. Therefore, one can use eigenvalue computation algorithms to find the roots $t_i$, which will serve as nodes of the quadrature rule. Note that if the $\{p_j(t)\}_{j=1}^N$ are orthonormal, then the matrix $T$ is similar to a symmetric matrix 
\[
J = \begin{pmatrix}
\alpha_1 & \beta_1 & 0 & 0  & \dots & 0\\[1ex]
\beta_1 & \alpha_2 & \beta_2 & 0  & \dots & 0\\[1ex]
0 & \beta_2 & \alpha_3 & \beta_3 & \dots & 0\\[1ex]
\vdots & \vdots & \ddots & \ddots & \ddots & \vdots\\[1ex]
0 & \ldots & 0 & \beta_{N-2} & \alpha_{N-1} & \beta_{N-1} \\[1ex]
0 & \ldots & \ldots & 0 & \beta_{N-1} & \alpha_N \\
\end{pmatrix},
\]
with $\alpha_j = -\frac{b_j}{a_j}$, $\beta_j = \sqrt{\frac{c_{j+1}}{a_j a_{j+1}}}$. In particular, the computation of the quadrature nodes from the entries of $J$ is well conditioned. Using the Christoffel-Darboux identity, Wilf~\cite{Wilf78:MathPhisicalSciences} deduces that
\[
w_i \mathfrak{p}(t_i)^\top \mathfrak{p}(t_i) = 1,
\]
which leads to explicit formulas for the computation of the weights $w_i$ of the quadrature rule. 

For the special case we are interested in, namely the case of Gegenbauer polynomials with weight function $\omega_{\alpha}(t) = (1-t^2)^{\alpha-\frac{1}{2}}$, the entries of $J$ for $j=1,\ldots, N$ are specified by 
\[
\alpha_j = 0, \quad \beta_j = \frac{1}{2}\sqrt{\frac{j(j+2\alpha-1)}{(j+\alpha-1)(j+\alpha)}}.
\]
allowing explicit computations. We further stress the key point: the Golub-Welsh approach allows us to compute the quadrature nodes and weights we require via eigenvalue algorithms on \emph{symmetric matrices}. This is very useful because the numerical behavior of symmetric eigenvalue computations is vastly superior to the general problem of root-solving (see~\cite{TreBau97:NLA}*{Chapter V} for details).

\subsubsection{Zonal polynomials}
The Gegenbauer polynomials introduced in Section~\ref{sec: Geg_quad} play a fundamental role in harmonic analysis on spheres and in the computation of convolutions because they can be used to explicitly describe the Christoffel-Darboux (or reproducing) kernel of the uniform measure on the sphere~\cite{ABR01:HarmonicFunctionTheory}*{Proposition 5.27, Theorem 5.29}.
More precisely, if $Z_d(t)$ is the normalized Gegenbauer polynomial defined as
\[ 
Z_d(t):=\frac{{\rm dim}(\mathcal{H}_d)}{\mu(S^{n-1}) C_d^{(\frac{n-2}{2})}(1)}C_d^{(\frac{n-2}{2})}(t),
\]
then for every $x\in S^{n-1}$ and $f\in L^2(S^{n-1},\mu)$ the equality
\[\int_{S^{n-1}} Z_d(\langle x,y\rangle)f(y)d\mu(y) = f_d(x)\]
holds where $f_d(x)$ denotes the $d$-th harmonic component of $f$. This property has many applications and in particular implies the remarkable Funk-Hecke formula below. To ease notation we will write $Z_d(x,y):=Z_d(\langle x,y\rangle)$ and refer to these functions as \emph{zonal polynomials}.

\begin{proposition}[Funk-Hecke formula]\label{Cor: Funk-Hecke}
If $u(t)=\sum_{j=0}^d\lambda_j Z_j(t)$ is the unique expression of a polynomial $u(t)$ as a combination of normalized Gegenbauer polynomials, and $f=\sum_{j=0}^{\infty}f_j$ is the harmonic decomposition of $f$, then
\[
\int_{S^{n-1}} u(\langle x,y\rangle)f(y){\rm d}\mu(y) = \sum_{j=0}^d \lambda_j f_j(x).
\]
\end{proposition}

\subsection{Uniform approximation of convolutions and harmonic expansions}

Given a quadrature rule $(Y,W_Y)$ on $S^{n-1}$ and a univariate polynomial $u(t)$ which is nonnegative on $[-1,1]$
we define the \emph{approximate convolution} $\widetilde{T_u(f)}$ by the formula
\[\widetilde{T_u(f)}(x):=\sum_{y_i\in Y} W_Y(y_i)u(\langle x,y_i\rangle) f(y_i).\]
The main results of this section are quantitative estimates guaranteeing that the approximate convolution and the true convolution are \emph{uniformly close} when the quadrature is exact on forms of sufficiently high degree and $f$ is Lipschitz continuous. 

In particular we obtain an algorithm to estimate the $d$-th harmonic component of such a function in $S^{n-1}$ with a uniform error bound of $\epsilon$ and prove in Corollary~\ref{cor:algo1} that this problem can be solved in polynomial time in $d$ for each fixed dimension $n$.  This provides a way to explicitly create polygauge/polyradial approximations of very complicated starbodies. The following Proposition proves the correctness of Algorithm \ref{alg:approx_fourier}, as stated in the introduction, and quantifies the constants $c_{n,m}$ appearing in the description of the algorithm.
\begin{proposition}\label{thm:uniform_harmonic} 
Let $d$ be a positive integer and let $f\in L^2(S^{n-1},\mu)$ be a function. If for some $\varepsilon>0$ there exists a polynomial $p_m$ of degree $m$ with
\[
\|f-p_m\|_{\infty}<\frac{\varepsilon}{2\sqrt{\mu(S^{n-1})\dim(\mathcal{H}_d)}},
\]
and $(X,W_X)$ is a quadrature rule exact in degree $\max(m+d,2d)$, then the polynomial
\[
\widetilde{f}_d(x):=\sum_{z_i\in X} W_X(z_i)f(z_i)Z_d(\langle x,z_i\rangle)
\]
of degree $\leq d$ is a uniform approximation to the $d$-th harmonic component $f_d$ of $f$, satisfying the inequality
$\|\widetilde{f}_d-f_d\|_{\infty}<\varepsilon$.
\end{proposition}
\begin{proof} 
Since the quadrature rule $(X,W_X)$ is exact in degree $m+d$, the equality
\[
\int_{S^{n-1}} p_m(z)Z_d(x,z) \d\mu(z) = \sum_{z_i\in X} W_X(z_i)p_m(z_i)Z_d(x,z_i)
\]
holds for every $x\in S^{n-1}$.
Therefore, for every $x\in S^{n-1}$ 
the triangle inequality yields
\begin{align*}
    |\widetilde{f}_d(x)-f_d(x)| &\leq \left|\widetilde{f}_d(x)-\!\sum_{z_i\in X} \! W_X(z_i)p_m(z_i)Z_d(x,z_i)\right|+\left|f_d(x)-\!\int_{S^{n-1}} \! p_m(z)Z_d(x,z)\d\mu(z)\right| \\
    &= \left|\sum_{z_i\in X} \! W_X(z_i)(f(z_i)-p_m(z_i))Z_d(x,z_i)\right|+\left|\int_{S^{n-1}} \!\!(f(z)-p_m(z))Z_d(x,z)\d\mu(z)\right|.
\end{align*}
If $\|f-p_m\|_{\infty}<\eta$ for some $\eta>0$, then by the triangle inequality we get 
\begin{align*}
    |\widetilde{f}_d(x)-f_d(x)| &\leq \eta\left(\sum_{z_i\in X} W_X(z_i)|Z_d(x,z_i)|+\int_{S^{n-1}} |Z_d(x,z)|\d\mu(z)\right) \\
    &\leq \eta\sqrt{\mu(S^{n-1})} \left( \left( \sum_{z_i\in X} W_X(z_i)Z_d(x,z_i)^2 \right)^{\frac{1}{2}} + \left( \int_{S^{n-1}} Z_d(x,z)^2 \d\mu(z)\right)^{\frac{1}{2}} \right),
\end{align*}
where the last inequality follows by Cauchy-Schwarz.
Using the fact that the quadrature rule is exact in degree $2d$, and the reproducing property of the zonal polynomials, we conclude that 
\[
|\widetilde{f}_d(x)-f_d(x)| \leq \eta\sqrt{\mu(S^{n-1}) 4 Z_d(x,x)} = 2\eta\sqrt{\mu(S^{n-1})\dim(\mathcal{H}_d)}.
\]
Therefore, by taking $\eta<\frac{\varepsilon}{2\sqrt{\mu(S^{n-1}) \dim(\mathcal{H}_d)}}$ we get $\|\widetilde{f}_d-f_d\|_{\infty}<\varepsilon$, as claimed.
\end{proof}

The hypotheses of Theorem \ref{thm:uniform_harmonic} are satisfied in particular for Lipschitz functions.
\begin{theorem} \label{cor:algo1}
Let $\varepsilon>0$ be given. If $f\in L^2(S^{n-1},\mu)$ is a continuous function with Lipschitz constant $\kappa$, then Algorithm \ref{alg:approx_fourier} gives an $\varepsilon$-uniform polynomial approximation to the $d$-th harmonic component of $f$, with time complexity polynomial in $d$, in fixed dimension $n$.
\end{theorem}
\begin{proof} 
By Theorem \ref{thm:NewmannShapiro_operator}, there exists a polynomial $p_m$ of degree $m$ with $\|f-p_m\|_\infty \leq \frac{\pi (n-2)}{\sqrt{2}}\, \frac{\kappa}{m}$. In particular, we can choose $m = \left\lceil \frac{\pi (n-2) \kappa \sqrt{2\mu(S^{n-1})\dim(\mathcal{H}_d)}}{\varepsilon}\right\rceil$ and build a quadrature rule exact in degrees $m+d$ and $2d$ in polynomial time in $d$. The resulting harmonic approximation is $\varepsilon$-uniform by Theorem~\ref{thm:uniform_harmonic}.
\end{proof}
As a consequence of Corollary \ref{cor:algo1}, we get $c_{\kappa,d,n,m} = \frac{\pi(n-2)\kappa}{m} \sqrt{2\mu(S^{n-1})\dim(\mathcal{H}_d)}$ for the uniform bound in Algorithm \ref{alg:approx_fourier}.

\begin{remark}
Algorithm \ref{alg:approx_fourier} is polynomial in fixed dimension $n$, but it scales exponentially as the dimension increases. In fact, since it relies on quadrature rules there is the fundamental constraint that no quadrature rule having less than $\binom{n-1+k}{k}$ points is exact in degree $2k$ on $S^{n-1}$. In practice, our approach is computationally tractable only for spheres of small dimension. Note, however, that the algorithm does not require any optimization but only numerical computation.
\end{remark}

Putting together Algorithm \ref{alg:approx_fourier} and the Funk-Hecke formula from Corollary \ref{Cor: Funk-Hecke}, we obtain Algorithm \ref{alg:approx_mollified} which provides an efficient approximation of a Lipschitz function $f$ following Theorem \ref{thm:NewmannShapiro_operator}. Figure \ref{fig:low_deg_cube} displays polyradial bodies of degrees $5$, $10$, $20$ obtained from mollified approximations (using Algorithm \ref{alg:approx_mollified}) of the radial function of the cube.
\begin{algorithm}[ht!]
    \caption{Compute a mollified approximation of $f$ by a polynomial of degree $d$}
    \label{alg:approx_mollified}
    \textsc{Input:} $f \in L^2(S^{n-1},\mu)$ with spherical Lipschitz constant $\kappa$, $d=2k>0$, 
    $m>0$.\\
    \textsc{Output:} $\widetilde{T_{u_{2k}}f}_d$, so that $\|f - \widetilde{T_{u_{2k}}f}_d\|_\infty<c'_{\kappa,d,n,m}$ for some $c'_{\kappa,d,n,m}>0$.
    \begin{algorithmic}[1]
    \For{$j = 1,\dots, d$}
        \State $\lambda_j \gets $ coefficient so that $u_{2k}(t) =$ \eqref{eq:u_k_optimal} $= \sum_{j=1}^d \lambda_j Z_j(t)$ \label{step_algo:coeffs_u}
        \State $\widetilde{f}_j \gets$ output of Algorithm \ref{alg:approx_fourier} with input $f$, $j$, $m$
    \EndFor    
    \State \Return $\widetilde{T_{u_{2k}}f}_d \gets \sum_{j=1}^d \lambda_j \widetilde{f}_j$
    \end{algorithmic}
\end{algorithm}
\begin{figure}[ht]
    \centering
    \includegraphics[height=3.5cm]{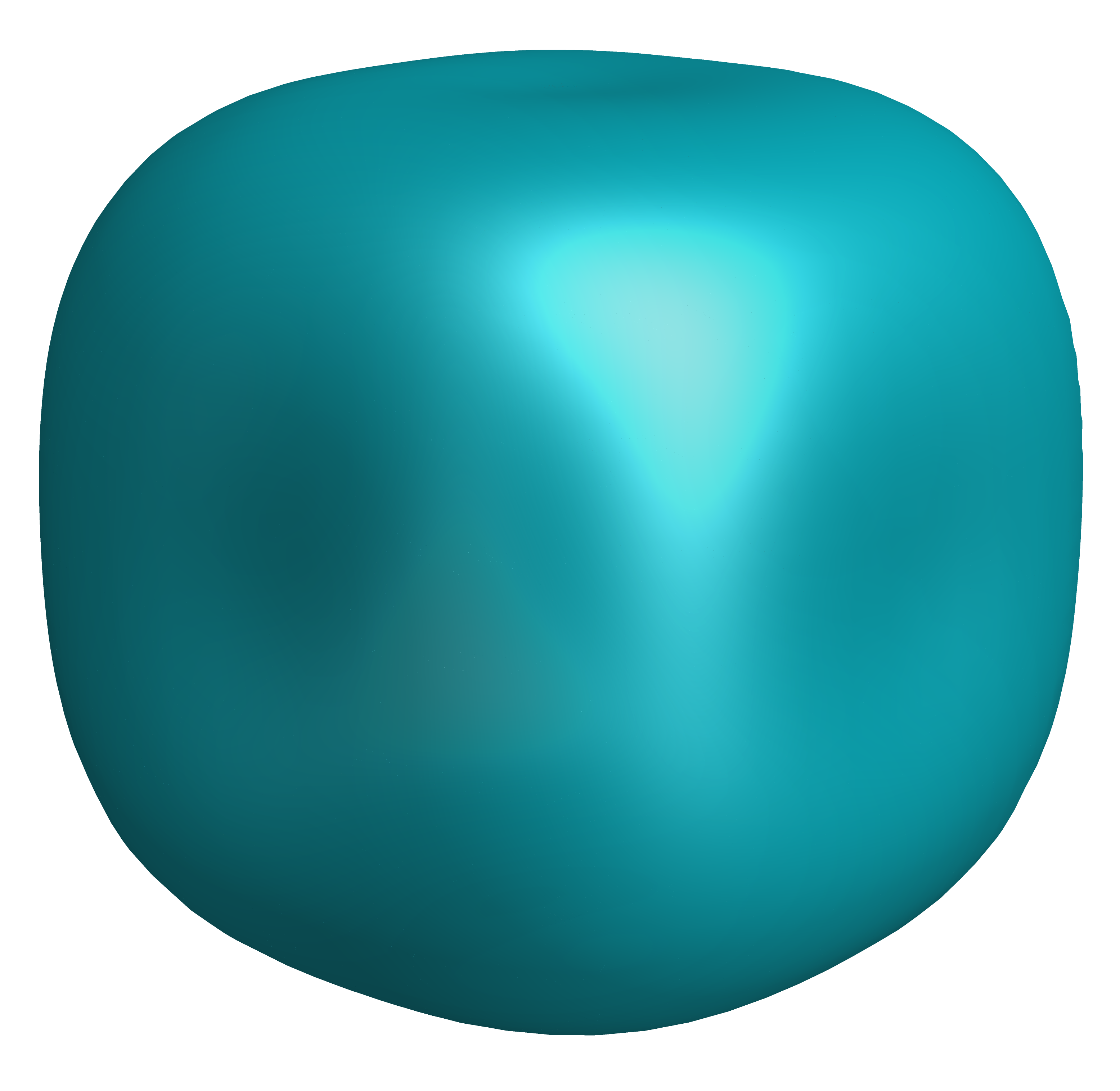}
    \;
    \includegraphics[height=3.5cm]{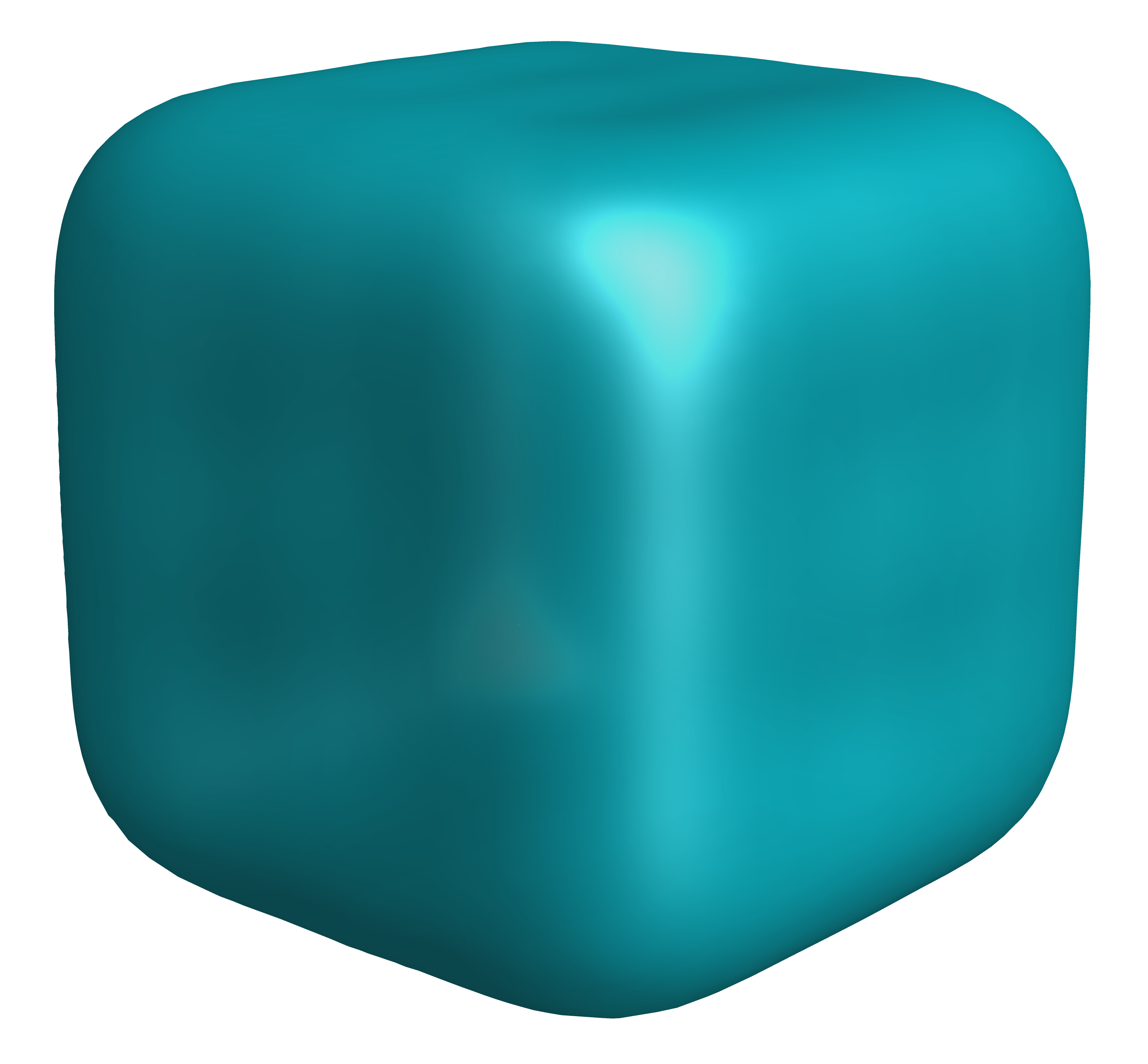}
    \;
    \includegraphics[height=3.5cm]{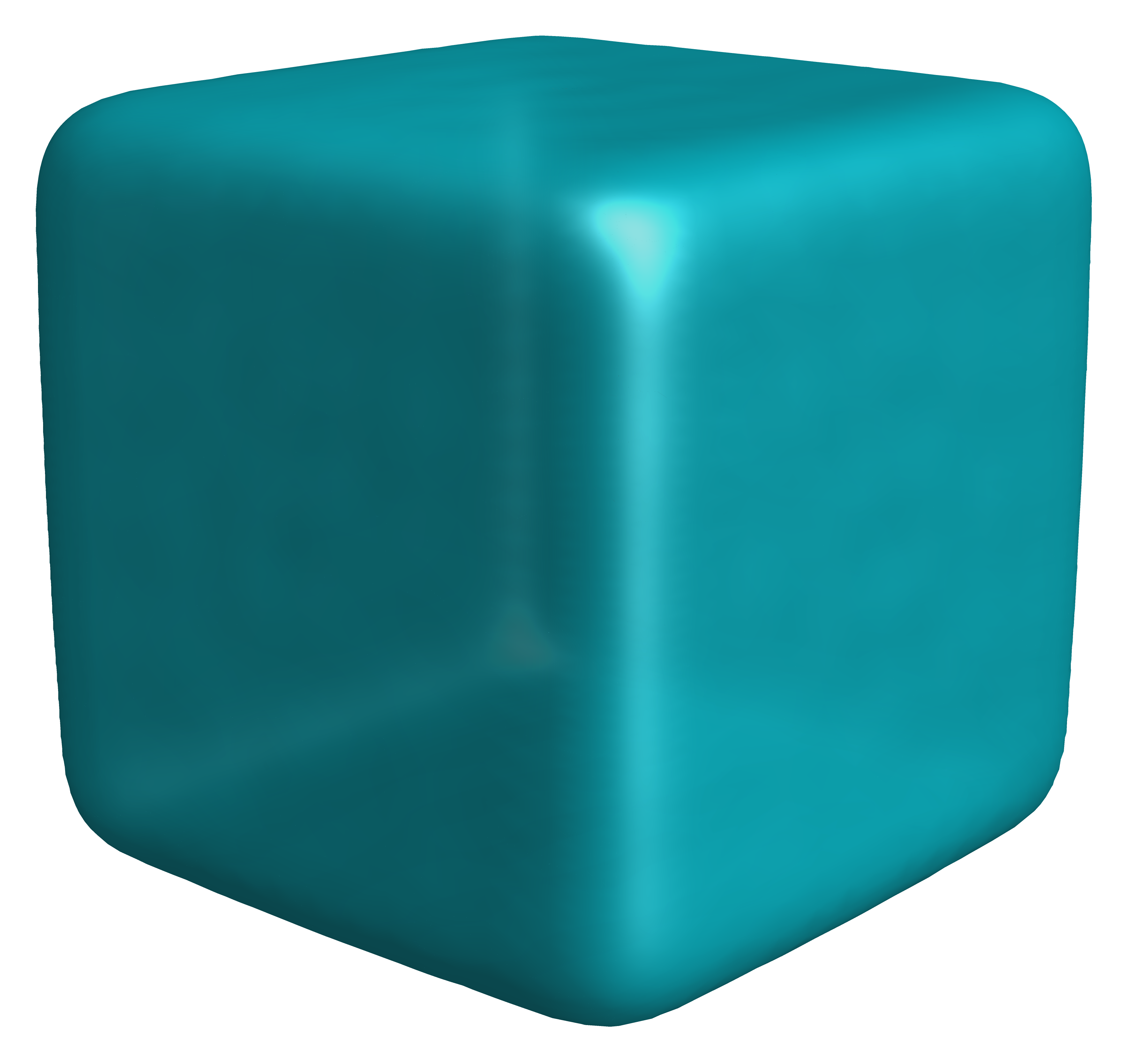}
    \caption{Polyradial bodies of degrees $5$, $10$, $20$ from left to right, computed using Algorithm \ref{alg:approx_mollified} from the radial function of the cube.}
    \label{fig:low_deg_cube}
\end{figure}
Note that in Step \ref{step_algo:coeffs_u}, we assume that the decomposition of $u_{2k}$ as a linear combination of zonal polynomials is given. This can be achieved by a simple change of basis for univariate polynomials of degree $2k$.

Our final theorem shows that the approximate convolutions of nonnegative functions, an in particular the gauge or radial functions of polystar approximations are sums-of-squares in $\RR[S^{n-1}]$.

\begin{theorem} \label{polySOS} Assume $u(t)$ is a univariate polynomial which is nonnegative on $[-1,1]$. If $f$ is nonnegative at the points of the quadrature rule $(Y,W_Y)$ then the output $\widetilde{T_{u}(f)}$ of Algorithm \ref{alg:approx_mollified} is a sum of squares in $\RR[S]$.
\end{theorem}
\begin{proof} Since $u(t)$ is nonnegative in $[-1,1]$ there exist univariate sums of squares $s_0(t)$, $s_1(t)$ such that $u(t)=s_0(t)+(1-t^2)s_1(t)$ so the approximate convolution at a point $x\in S^{n-1}$ is given by
\[\widetilde{T_{u}(f)}(x)=\sum_{y_i\in Y} \left(s_0(\langle x,y_i\rangle) + (1-\langle x,y_i\rangle^2)s_1(\langle x,y_i\rangle)\right)W_Y(y_i)f(y_i)\]
We claim that for any $z\in S^{n-1}$ the function $\phi_z(x):=1-\langle x,z\rangle^2$ is a sum of squares in $\RR[S^{n-1}]$. This property implies the desired conclusion because the sums-of-squares are closed under products and conical combinations.
To verify the claim suppose $z,q_2,\dots, q_n$ are any orthonormal basis of $\RR^n$. By the pythagorean theorem we have that for any $x\in \RR^n$ 
\[\|x\|^2 = \langle x,z\rangle^2 +\sum_{j=2}^n \langle x,q_j\rangle^2,\]
hence $1-\langle x,z\rangle^2=\sum_{j=2}^n \langle x,q_j\rangle^2$ is a sum-of-squares in the coordinate ring $\RR[S^{n-1}]$.
\end{proof}

\section{Applications and Examples}
\label{Sec: polystar_applications}
In this Section we discuss two applications of polystar bodies. The philosophy behind the whole approach is that certain invariants are easier to compute on polystar bodies than on general starbodies. If these invariants depend continuously on the body $E$ then the easily computable invariant of a polystar approximation of $E$ gives a good approximation for the true invariant of $E$. To illustrate this philosophy we focus on two concrete problems:
\begin{enumerate}
    \item The computation of intersection bodies and largest slices. Theorem~\ref{thm: polystar_intersection_body} shows that the intersection body of a polyradial body is itself polyradial. Section~\ref{subsec:intersectino_bodies} contains images of intersection bodies obtained via this mechanism, and the estimate of the largest volume slice of the associated starbodies. The examples show that the method is applicable beyond the highly structured sets such as polytopes where other methods are available~\cite{BBMS22:IntBodiesPolytopes}.

    \item\label{item:width} The computation of width. Theorem~\ref{polySOS} shows that the width of a polygauge body can be computed via a sequence of semidefinite programs of increasing degrees.
    Section \ref{subsec:width} computes numerically the width of some convex bodies using approximations of the gauge function of their polar bodies.
\end{enumerate}

In both problems, we are interested in maximizing a positive function on the sphere, approximating it with a polynomial. In practice, this can be done using  sum-of-squares hierarchies \cites{Reznick95:Hilbert17,FF21:SOSSphere}. The approximation of the gauge/radial function on $S^{n-1}$ by a polynomial of degree $d$ has error of the order $\frac{n}{d}$, while the approximation of the maximum of a polynomial of degree $d$ using the sum-of-squares hierarchy at level $\ell$ has error $\frac{n^2}{\ell^2}$ when $d\leq 2n$ \cite{FF21:SOSSphere}, and error $\frac{n}{\ell}$ when $d > 2n$, for $\ell$ large enough (original result due to \cite{Reznick95:Hilbert17}, then elaborated in \cites{Faybusovich04:GlobalOpt,DohWeh12:ConvergenceSDP}).

Our algorithms are implemented in \texttt{Python} \cite{PythonV3} and available at 
\begin{center}
    \url{https://github.com/ChiaraMeroni/polystar_bodies}.
\end{center}
For the maximization of high degree polynomials and for plots, we rely on \texttt{Mathematica} \cite{MathematicaV14_1}.

\subsection{Polyradial intersection bodies}\label{subsec:intersectino_bodies}

Recall that the intersection body $IL$ of a starbody $L\subset\RR^n$ is the starbody with radial function 
\[
\rho_{IL}(x) = \vol (L \cap x^\perp ) = \frac{1}{n-1} \int_{S^{n-1}\cap x^\perp} \rho_L(y)^{n-1} \d\mu(y).
\]

\begin{theorem} \label{thm: polystar_intersection_body} If $L\subseteq \RR^n$ is a polyradial body then the intersection body $IL\subseteq \RR^n$ is also polyradial.
\end{theorem}
\begin{proof} If $\rho_L$ is a regular function in $S^{n-1}$ then so is $\rho_L^{n-1}$ and in particular this function has a finite harmonic decomposition
$\rho_L^{n-1}(x)=\sum_{j=0}^M p_j(x)$.
The radial function of the intersection body can be expressed from $\rho_L^{n-1}(x)$ via the spherical Radon transform $\mathcal{R}$ in $L^2(S^{n-1},\mu)$ as
\[
\rho_{IL}(x)=\mathcal{R}\left( \frac{1}{n-1} \rho_L^{n-1} \right)(x)= \frac{1}{n-1}\sum_{j=0}^M c_j p_j(x),
\]
where the last equality follows from the linearity of $\mathcal{R}$ and the Funk-Hecke formula, for a certain set of coefficients $(c_j)_{j\in \mathbb{N}}$ depending only on the dimension $n$.
\end{proof}

\begin{remark} 
For $n=3$ (see \cite{Funk13}), the coefficients $c_j$ such that $\mathcal{R}(f) = \sum c_j f_j$ as in the proof of Theorem \ref{thm: polystar_intersection_body} satisfy $c_j=P_j(0)$, where $P_j$ is the Legendre polynomial of degree $j$ and in particular
\[
P_j(0) = \begin{cases}
    (-1)^\frac{j}{2} \frac{(j-1)!!}{j!!}, & j \hbox{ even},\\
    0, & j \hbox{ odd}.
\end{cases}
\]
For the sake of clarity, we provide in Algorithm \ref{alg:IB} explicit pseudocode for the computation of a polynomial approximation of the intersection body of $L\subset\RR^3$. The algorithm does not require any optimization but only numerical computation.
\end{remark}
\begin{algorithm}[ht!]
    \caption{Compute a mollified polynomial approximation of the intersection body of $L\subset\RR^3$}
    \label{alg:IB}
    \textsc{Input:} $\rho_L$ radial function of $L$, $d=2k>0$, $m>0$.\\
    \textsc{Output:} $(\widetilde{\rho_{IL}})_d$, the degree $d$ mollified polynomial approximation of the radial function of the intersection body $IL$ of $L$.
    \begin{algorithmic}[1]
    \For{$j = 2,\dots, d=2k$ even}
        \State $\lambda_j \gets $ coefficient so that $u_{2k}(t) =$ \eqref{eq:u_k_optimal} $= \sum_{j=1}^d \lambda_j Z_j(t)$
        \State $P_j \gets (-1)^\frac{j}{2} \frac{(j-1)!!}{j!!}$
        \State $\widetilde{f}_j \gets$ output of Algorithm \ref{alg:approx_fourier} with input $\frac{1}{2}\rho_L^2$, $j$, $m$
    \EndFor    
    \State \Return $(\widetilde{\rho_{IL}})_d \gets \sum_{j=2}^d P_j \lambda_j \widetilde{f}_j$ sum over even $j$
    \end{algorithmic}
\end{algorithm}

\subsubsection{Some intersection bodies in 3D}\label{sec:examples}
We exhibit in this section concrete examples that were computed using Algorithms \ref{alg:approx_fourier}, \ref{alg:approx_mollified} above on radial functions of three-dimensional starbodies. Figures \ref{fig:ex_cube}, \ref{fig:ex_cyl&co}, \ref{fig:ex_elliptope} give visual evidence to support the fact that the mollified approximation is an improvement over the simple truncation of Fourier expansion of a function. 
Moreover, we use our tools to approximate the intersection body of a starbody, via Algorithm \ref{alg:IB}. This allows us to estimate the slice of the original convex body with largest volume, via maximization of the (approximate) radial function of the intersection body:
\[
\max_{x\in S^{n-1}} \vol(L\cap x^\perp) = \max_{x\in S^{n-1}} \rho_{IL}.
\]
In practice, in our examples below we maximize the polynomial $(\widetilde{\rho_{IL}})_d$ from Algorithm \ref{alg:IB} using \texttt{Mathematica}'s command \texttt{NMaximize}. We report the estimated optimizer $x$ defining the largest volume slice $L\cap x^\perp$.

\begin{example}\label{ex:cube}
    Consider the standard cube $P\subset\RR^3$ with vertices $(\pm 1, \pm 1, \pm 1)$, displayed in Figure \ref{fig:ex_cube}, left. Its radial function reads 
    \[
    \rho_P(x) = \min_{i=1,2,3} \left\lbrace \frac{1}{|x_i|} \right\rbrace \quad \hbox{ for } x\in S^2,
    \]
    and it is Lipschitz continuous with Lipschitz constant $\kappa=1$. We run Algorithms \ref{alg:approx_fourier} and \ref{alg:approx_mollified} with $d=30$, $m=60$ and obtain the teal objects in Figure \ref{fig:ex_cube}, second and third from the left.
    In the case of a polytope, it is possible to compute the explicit description of its intersection body, see \cite{BBMS22:IntBodiesPolytopes}. Our Algorithm \ref{alg:IB}, with $d=30$, $m=60$, computes the degree $30$ polynomial $(\widetilde{\rho_{IL}})_{30}$ that we use to produce the yellow starbody in Figure \ref{fig:ex_cube}, right.
    We maximize numerically the radial function $(\widetilde{\rho_{IP}})_{30}$ to find the maximizer $(-0.7070, -0.7071, 0.0)$ with four digit precision, see Table \ref{tab:volume_slice}. This identifies the green slice in Figure \ref{fig:ex_cube} as the slice with largest volume. Indeed, one of the largest slices of the cube is given by $\frac{1}{\sqrt{2}}(-1,-1,0)$ \cite{Ball86:CubeSlicing}.
\end{example}
\begin{figure}[!ht]
\ifkeepslowthings
    \centering
    \includegraphics[height=3.5cm]{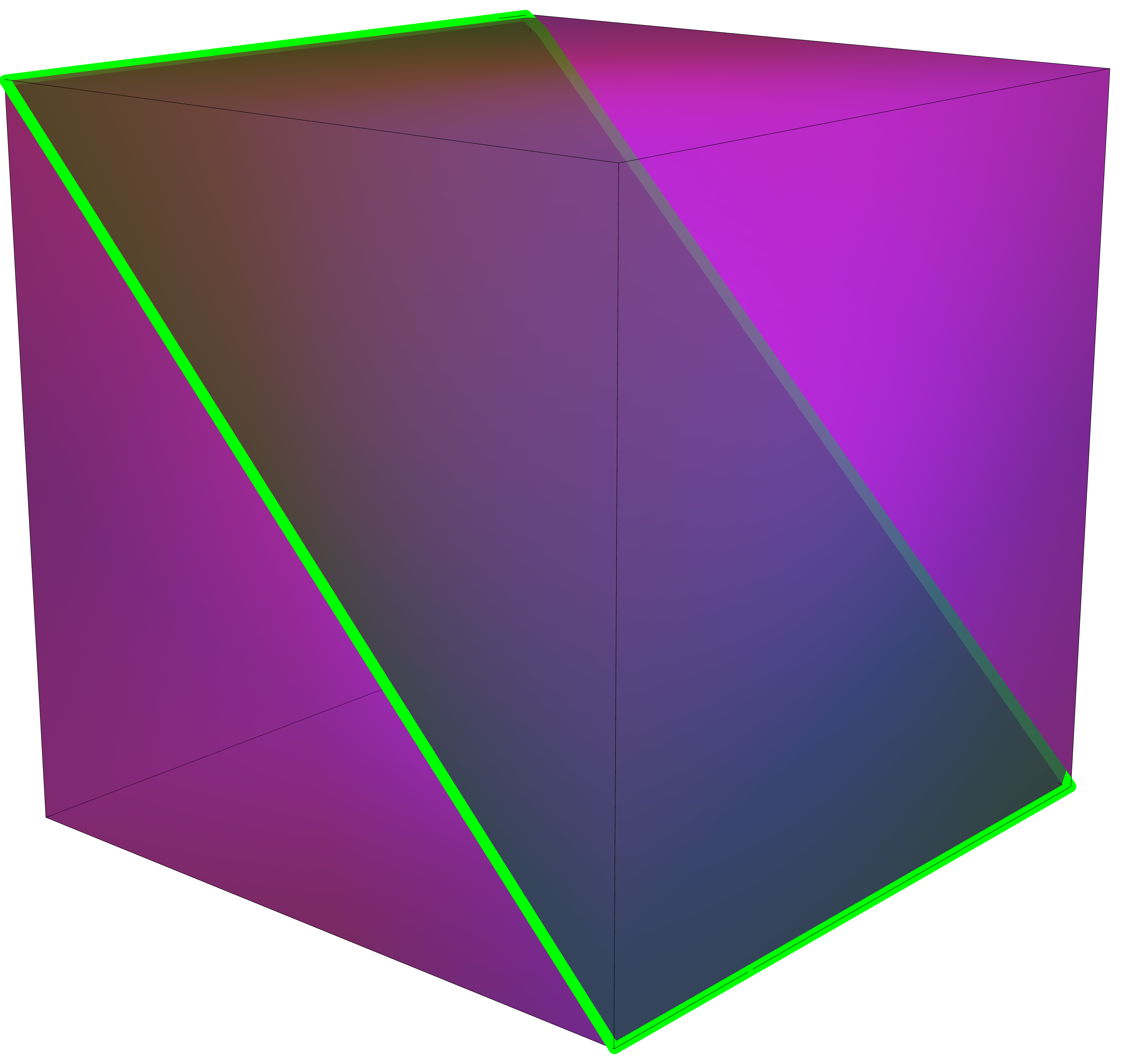}
    \;
    \includegraphics[height=3.5cm]{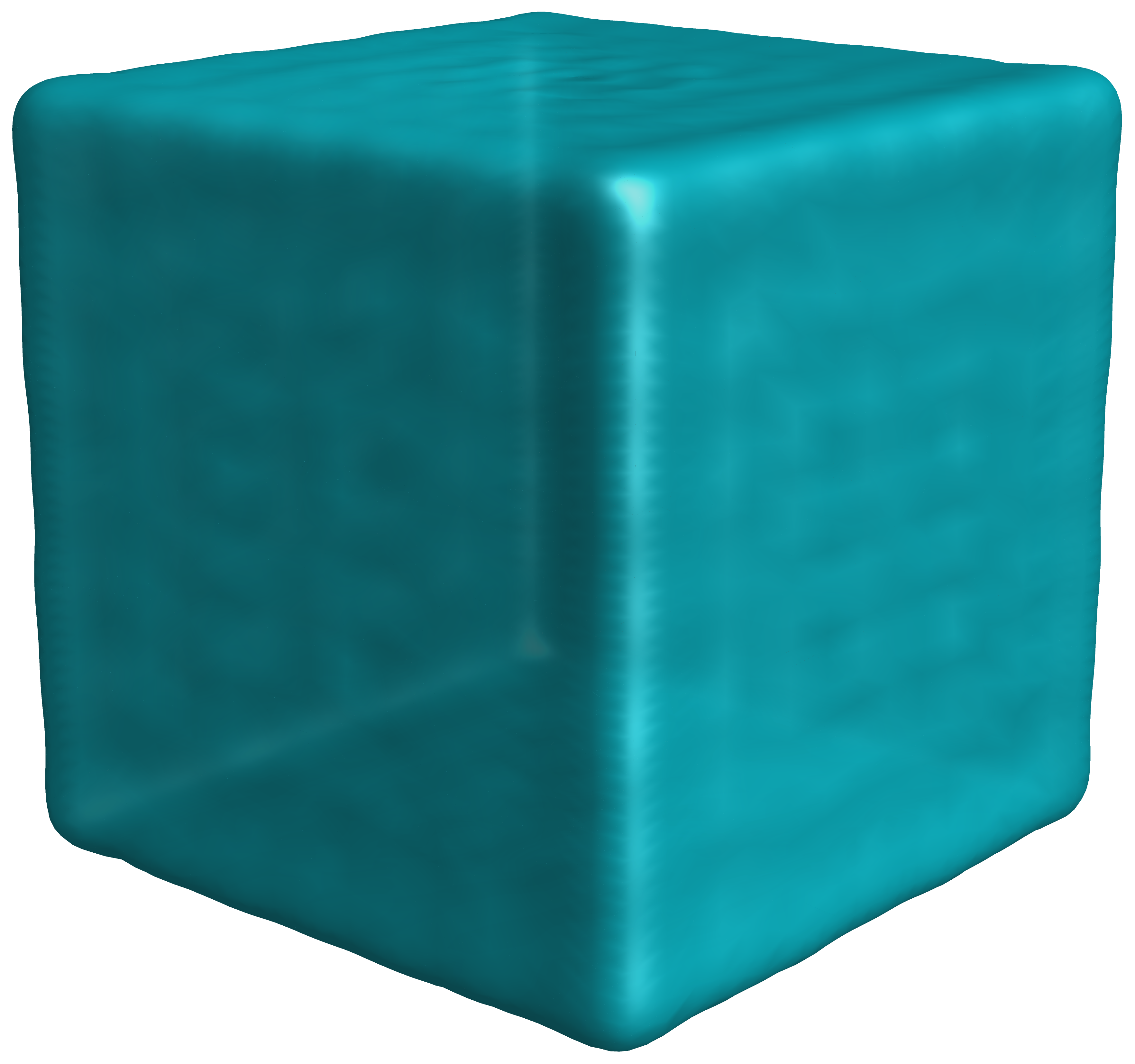}
    \;
    \includegraphics[height=3.5cm]{Figures/cubeWeighted.png}
    \;
    \includegraphics[height=3.5cm]{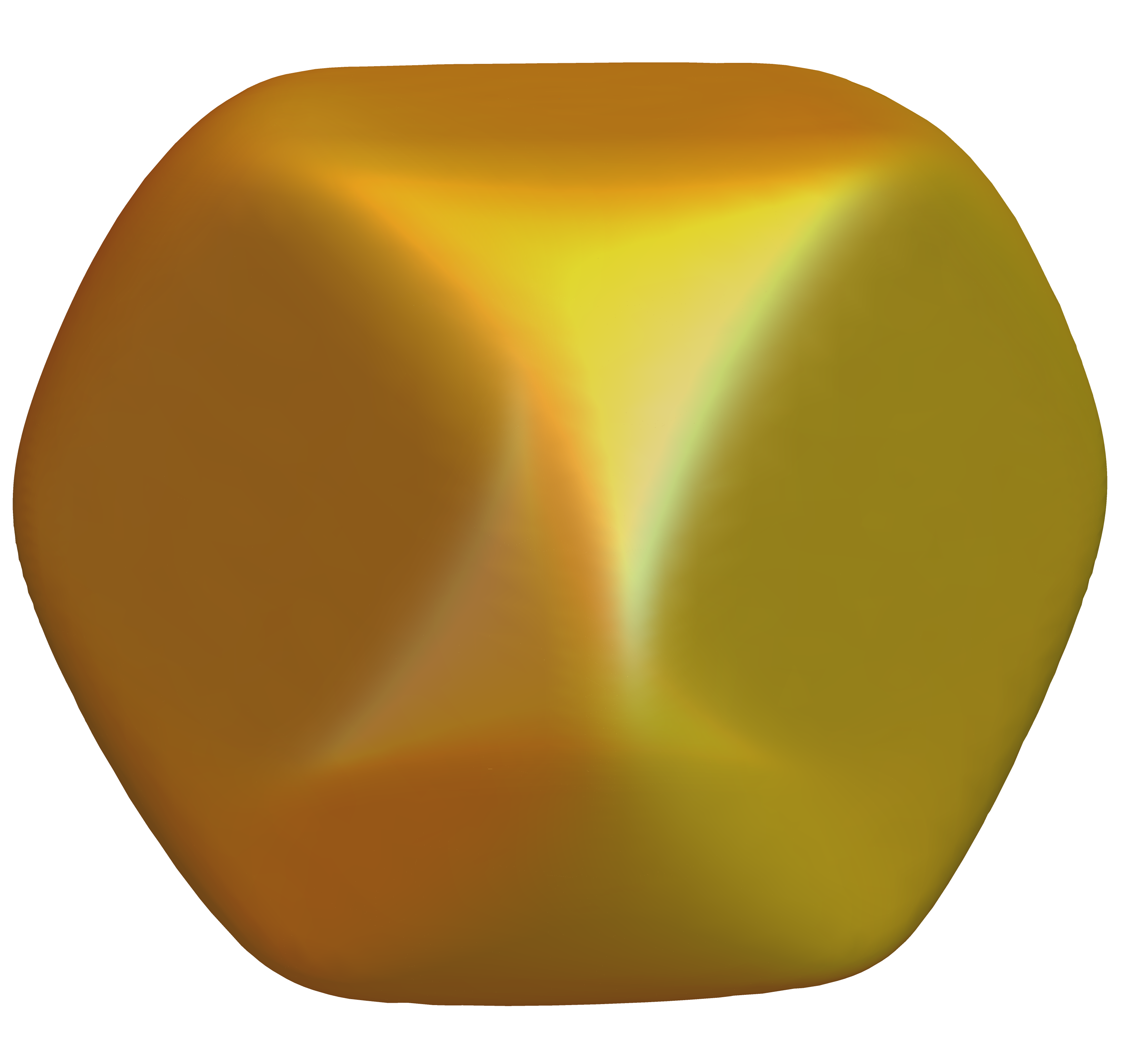}
    
\fi
    \caption{From left to right: the violet cube $P$ and its green slice with largest volume; the teal polynomial approximation of $P$, obtained via Algorithm \ref{alg:approx_fourier}; the teal \emph{mollified} polynomial approximation of $P$, obtained via Algorithm \ref{alg:approx_mollified}; the yellow polynomial approximation of the intersection body of $P$, via Algorithm \ref{alg:IB}. All approximating polynomials have degree $30$.}
    \label{fig:ex_cube}
\end{figure}

If polytopes have a very rich combinatorial structure that simplifies many operations, such as the explicit computation of their intersection body, more general starbodies or convex bodies are more complicated to treat via exact computation. This is where our approximation algorithm becomes essential. We illustrate it in a few nonpolytopal examples. We start with two important and renowned bodies related to cylinders.
\begin{example}\label{ex:cylinder}
    Let $C\subset\RR^3$ be the cylinder of revolution around the vertical $x_3$-axis, given by $\{x_1^2+x_2^2\leq 1, |x_3|\leq 1\}$, in Figure \ref{fig:ex_cyl&co}, first row, left. Its radial function can be expressed as
    \[
    \rho_C(x) = \begin{cases} 
     \frac{1}{| x_3 | } & x_3^2>x_1^2+x_2^2, \\
     \frac{1}{\sqrt{1-x_3^2}} & x_3^2<x_1^2+x_2^2. \\
    \end{cases}
    \]
    Its degree $30$ approximations obtained from Algorithms \ref{alg:approx_fourier} and \ref{alg:approx_mollified} are the teal bodies in Figure \ref{fig:ex_cyl&co}, first row, second and third objects, whereas its degree $30$ intersection body approximation is the yellow puffed cylinder in Figure \ref{fig:ex_cyl&co}, first row right. For all these computations we use quadrature rules exact in degree $60$. The numerical maximization of $(\widetilde{\rho_{IC}})_{30}$ provides the largest volume slice of $C$, see Table \ref{tab:volume_slice}. The numerical largest volume slice is shown in green in Figure \ref{fig:ex_cyl&co}, first row left.
\end{example}
\begin{figure}[!ht]
\ifkeepslowthings
    \centering
    \includegraphics[height=4cm]{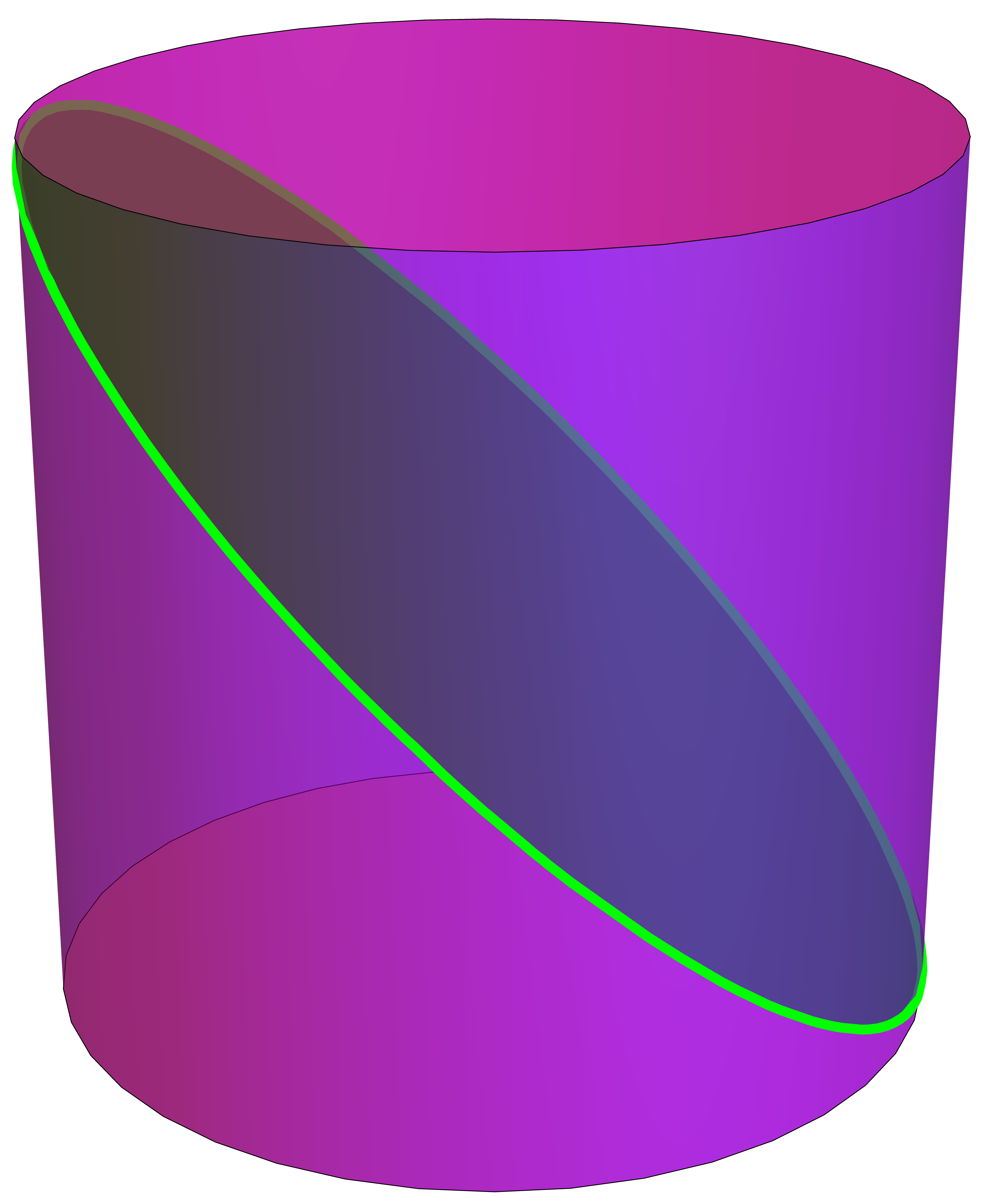}
    \quad
    \includegraphics[height=4cm]{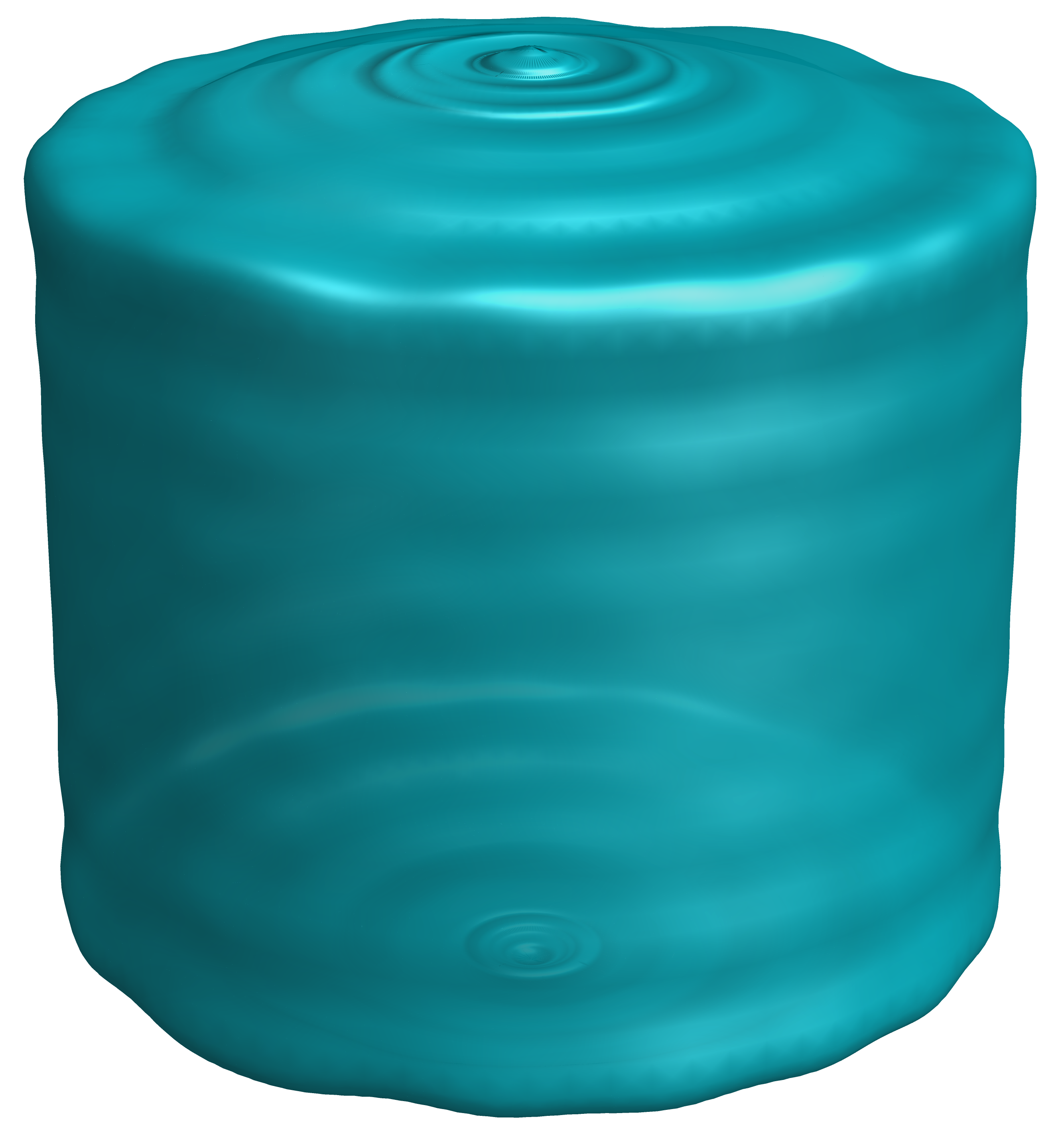}
    \quad
    \includegraphics[height=4cm]{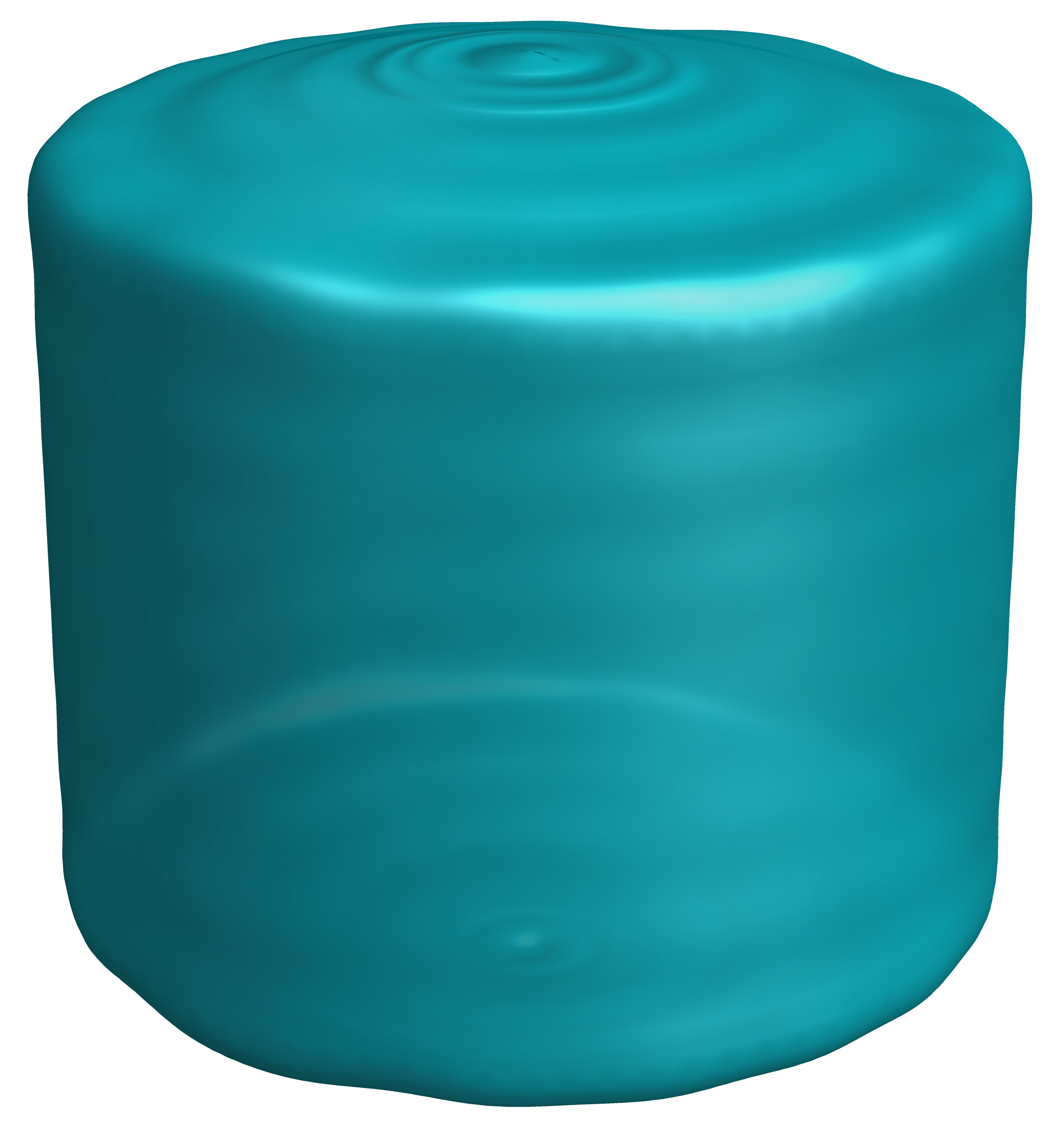}
    \quad
    \includegraphics[height=4cm]{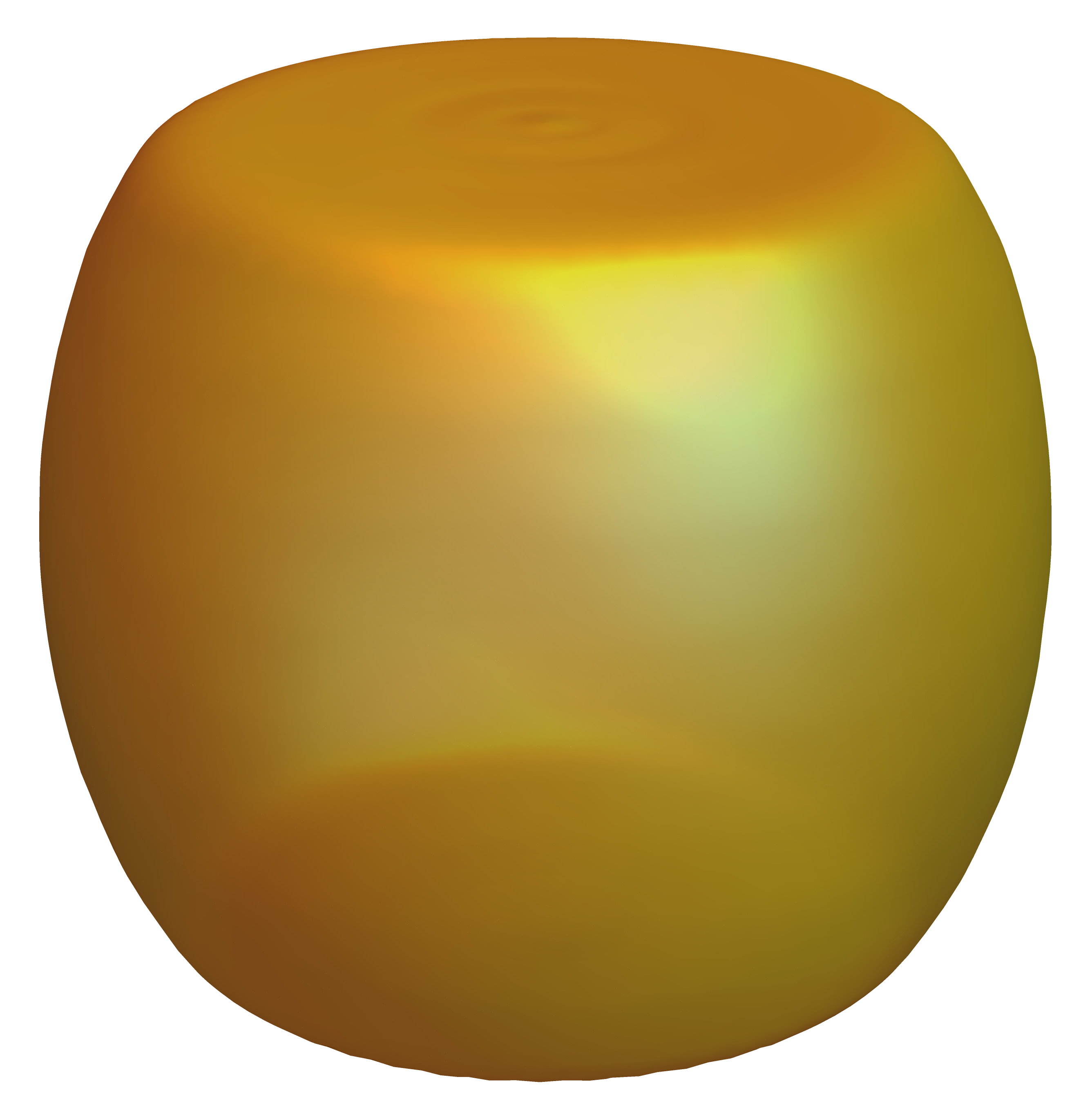} \\ 
    ~\\
    \includegraphics[height=4cm]{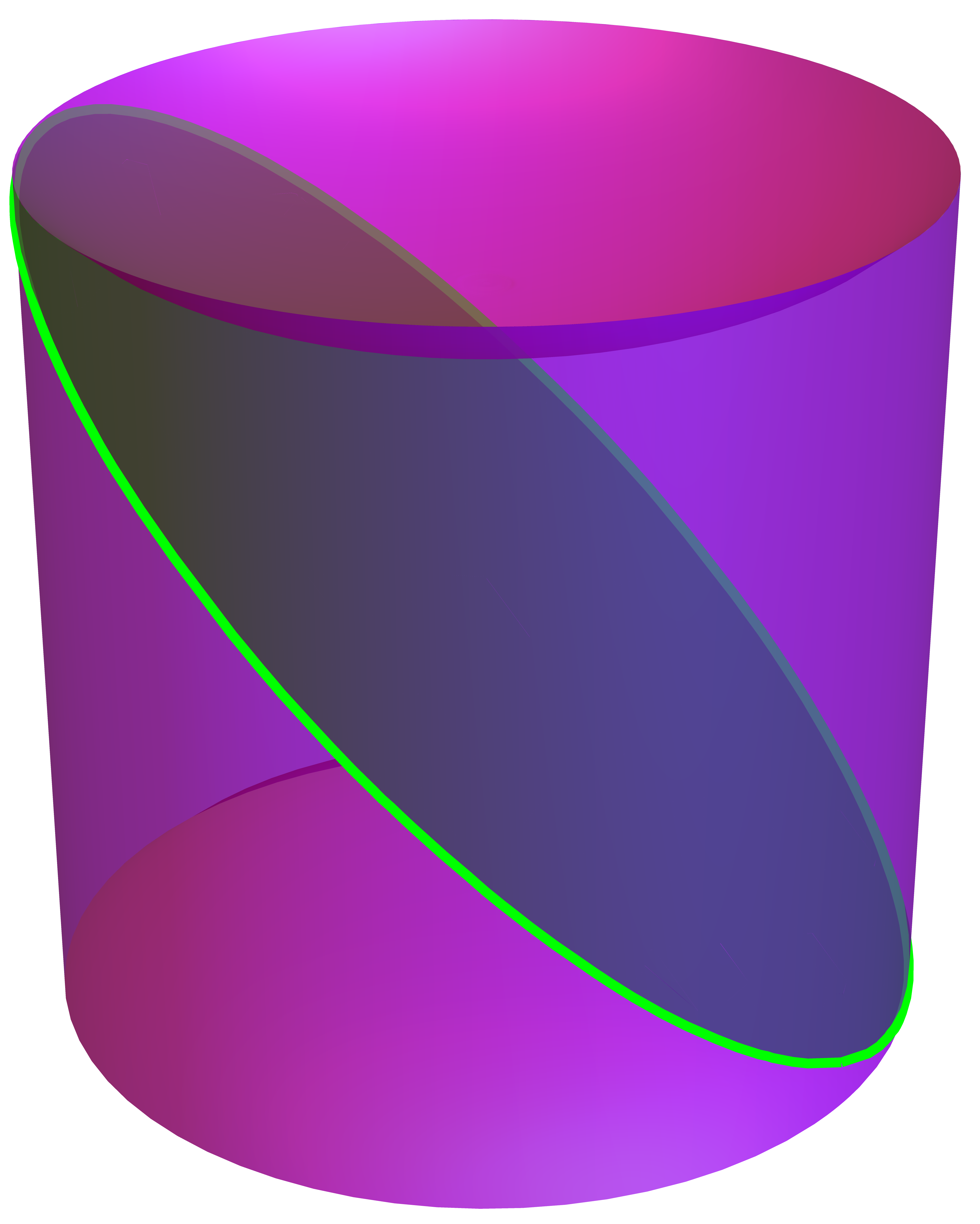}
    \quad \;
    \includegraphics[height=4cm]{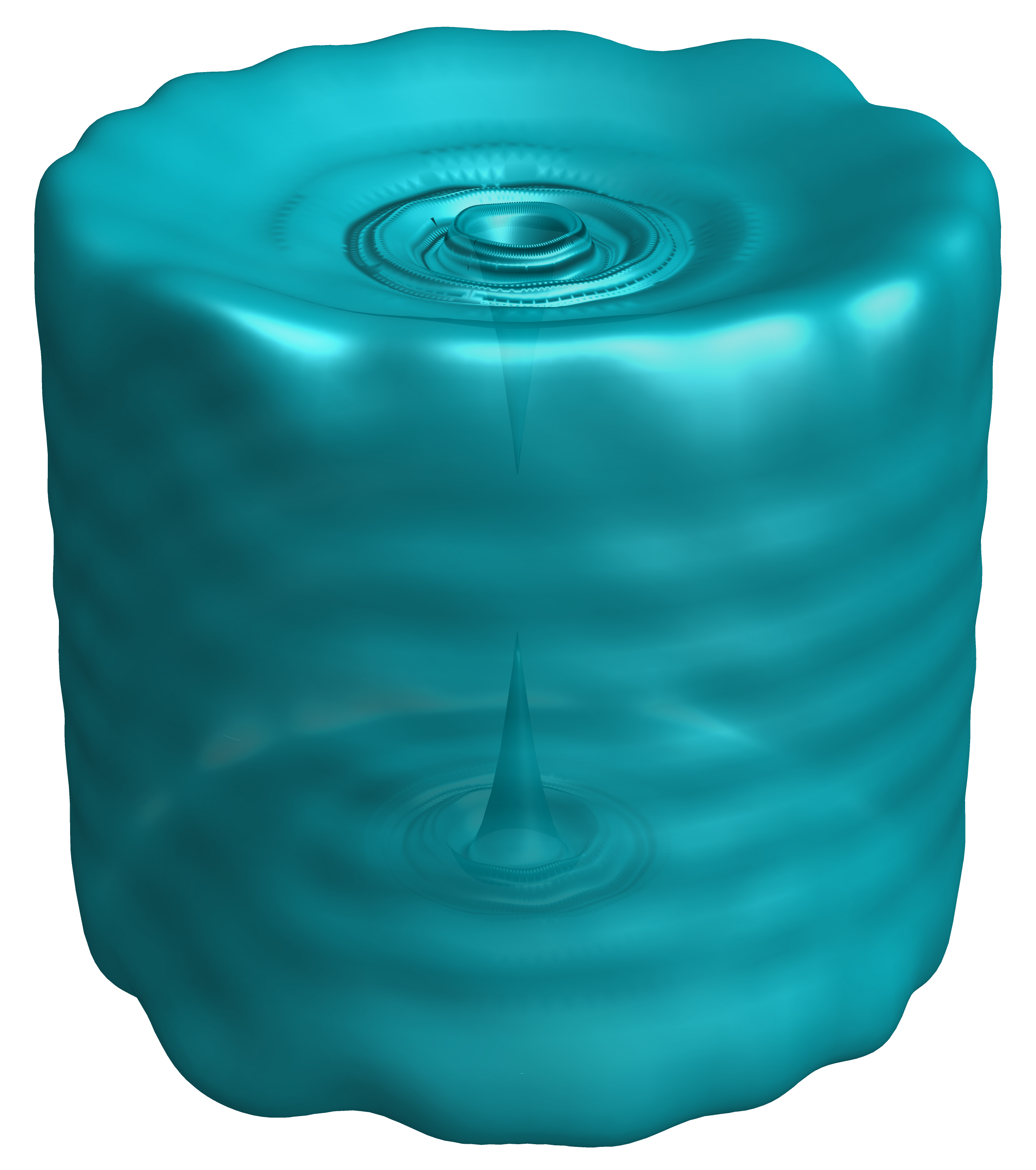}
    \quad \;
    \includegraphics[height=4cm]{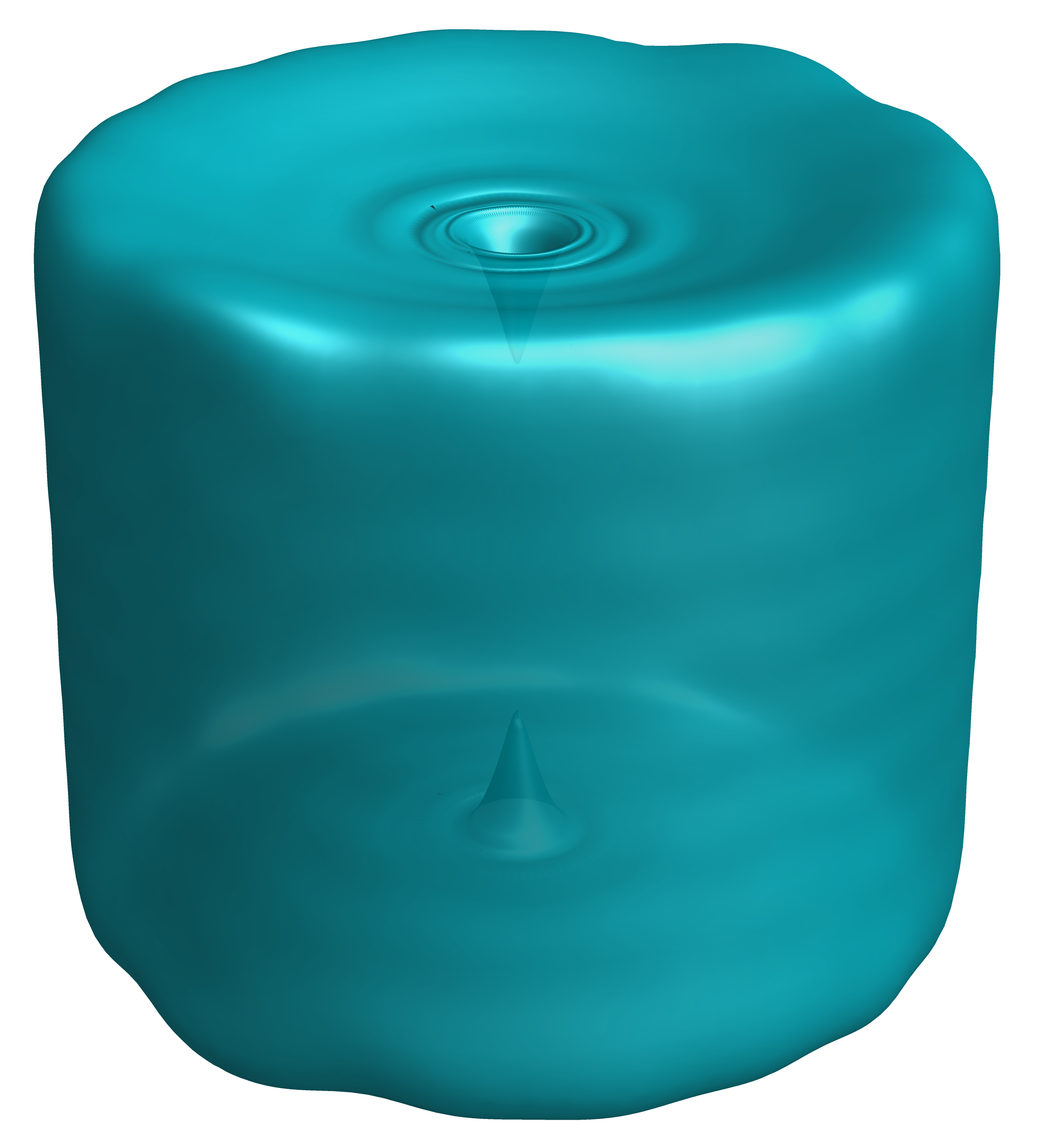}
    \quad \;
    \includegraphics[height=4cm]{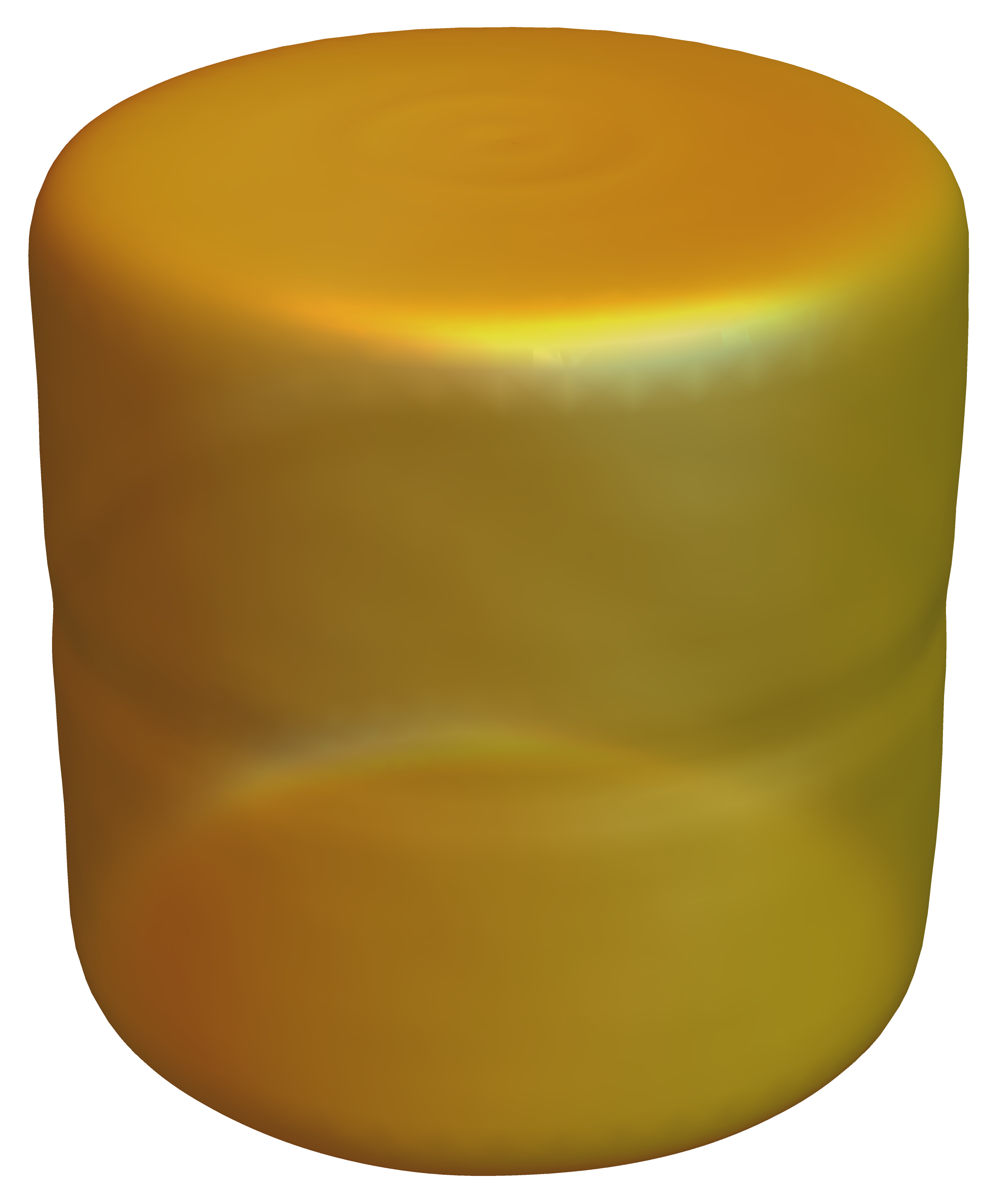}
\fi
    \caption{From left to right: the violet starbody $L$ and its green slice with largest volume; the teal polynomial approximation of $L$, via Algorithm \ref{alg:approx_fourier}; the teal \emph{mollified} polynomial approximation of $L$, via Algorithm \ref{alg:approx_mollified}; the yellow polynomial approximation of the intersection body of $L$, via Algorithm \ref{alg:IB}. All approximating polynomials have degree $30$.
    First row: $L=C$ is the cylinder from Example \ref{ex:cylinder}. Second row: $L=D$ is the dented tin can from Example \ref{ex:dented_tin_can}.}
    \label{fig:ex_cyl&co}
\end{figure}

\begin{example}\label{ex:dented_tin_can}
    Consider the \emph{dented tin can} $D\subset\RR^3$ from \cite{Gardner:GeometricTomography}*{Theorem 8.1.18}, a cylinder with concave top and bottom. The radial function of this nonconvex starbody is given by
    \[
    \rho_D(x) = \begin{cases} 
     \sqrt{\frac{1 - x_3^2 + x_3 \left(-1 + \sqrt{-1 + 2 x_3^2}\right)}{x_3 (-1 + x_3) \left(x_3^2 - \sqrt{-1 + 2 x_3^2}\right)}} & \sqrt{x_1^2+x_2^2}\leq x_3\leq 1, \\
     \frac{1}{\sqrt{1-x_3^2}} & x_3^2<x_1^2+x_2^2, \\
     \sqrt{\frac{-1 + x_3^2 + x_3 \left(-1 + \sqrt{-1 + 2 x_3^2}\right)}{x_3 (1 + x_3) \left(-x_3^2 + \sqrt{-1 + 2 x_3^2}\right)}} & -1\leq x_3 \leq -\sqrt{x_1^2+x_2^2}.
    \end{cases}
    \]
    This starbody is well-known since its intersection body is a cylinder, and bodies with such a property do not exist in higher dimension. The second row in Figure \ref{fig:ex_cyl&co} displays $D$, its degree $30$ standard (Algorithm \ref{alg:approx_fourier}) and mollified (Algorithm \ref{alg:approx_mollified}) Fourier approximation, and its intersection body (Algorithm \ref{alg:IB}). For all these computations we use quadrature rules exact in degree $60$. The numerical maximization of $(\widetilde{\rho_{ID}})_{30}$, giving the largest volume slice of $D$, returns the data in Table \ref{tab:volume_slice}. The numerical largest volume slice is shown in green in Figure \ref{fig:ex_cyl&co}, second row left.
\end{example}

We conclude our computations with the most famous, arguably, spectrahedron: the elliptope. To the best of our knowledge, there is no known explicit formula for its intersection body.
\begin{example}\label{ex:elliptope}
    Consider the elliptope $\mathcal{E}\subset\RR^3$, namely the set of positive semidefinite $3\times 3$ symmetric matrices with all-ones diagonal. Its radial function is given by 
    \[
    \rho_{\mathcal{E}} (x,y,z) = 
    \begin{cases}
    \frac{1}{6 x y z} \left( \sqrt{3}\sin{\frac{\theta}{3}} - \cos{\frac{\theta}{3}} - 1 \right),  & x y z < 0, \\
    \frac{1}{6 x y z} \left( 2 \cos{\frac{\theta}{3}} - 1 \right),  & x y z > 0,
    \end{cases}
    \]
    where $\theta = \arg{\left( 54 x^2 y^2 z^2-1+ i 6 \sqrt{3} \sqrt{x^2 y^2 z^2 \left( 1 - 27 x^2 y^2 z^2 \right)} \right)}$.
    The elliptope is the violet puffed tetrahedron in Figure \ref{fig:ex_elliptope}, left, its teal degree $30$ polynomial approximations via Algorithms \ref{alg:approx_fourier} and \ref{alg:approx_mollified} are the second and third objects, and its yellow approximated intersection body, which looks convex, is on the right. We use quadrature rules exact in degree $100$.
    Also in this case, we can maximize numerically the radial function $(\widetilde{\rho_{I\mathcal{E}}})_{30}$ to find an estimate of the largest volume slice of $\mathcal{E}$, see Table \ref{tab:volume_slice}. This computation suggests that the largest slices (Figure \ref{fig:ex_cyl&co}, left) are those containing two of the six nodes, giving rise to convex bodies delimited by a parabola and a segment.
\end{example}
\begin{figure}[ht]
    \ifkeepslowthings
        \centering
        \includegraphics[height=3.8cm]{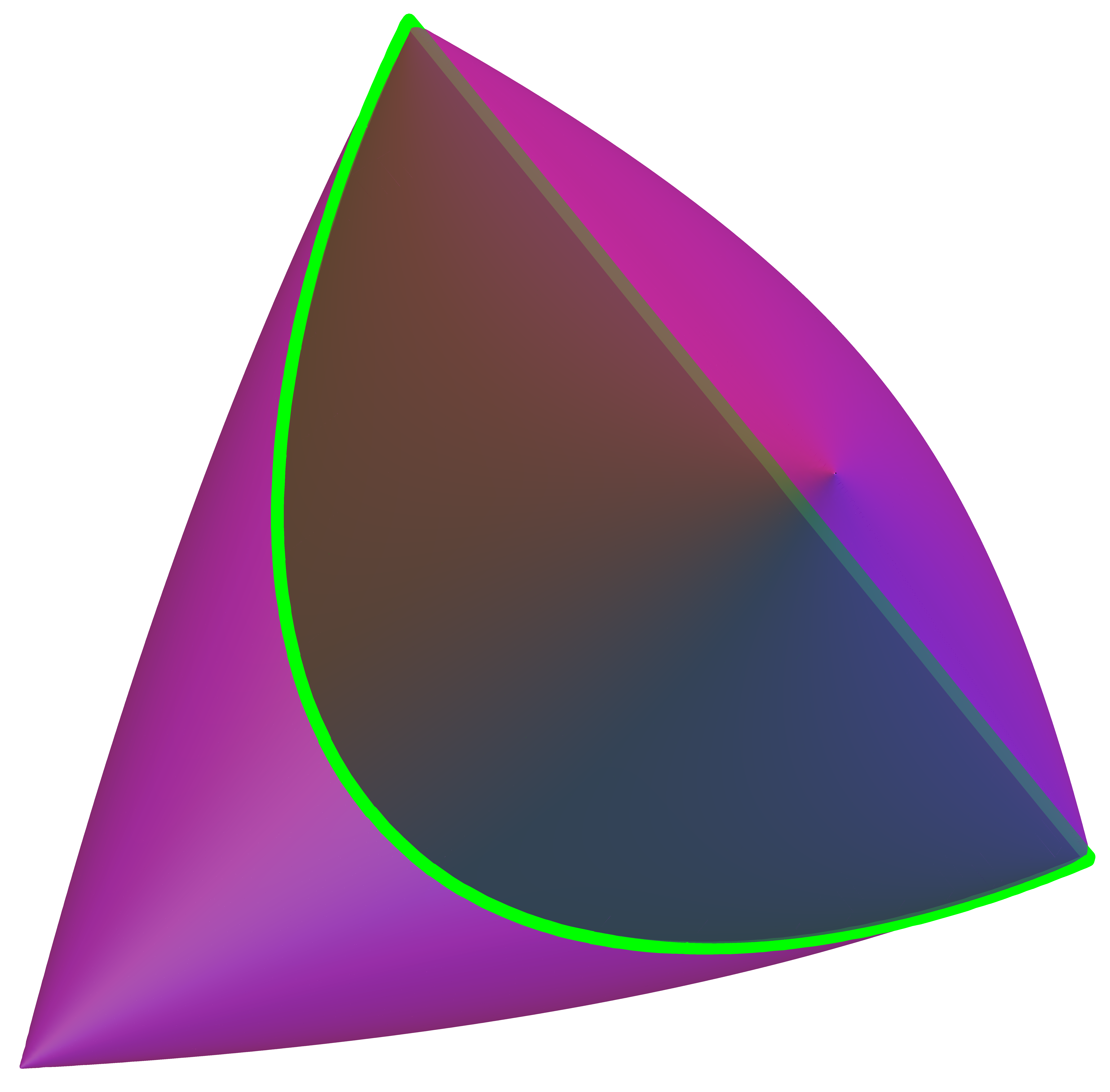}
        \includegraphics[height=3.8cm]{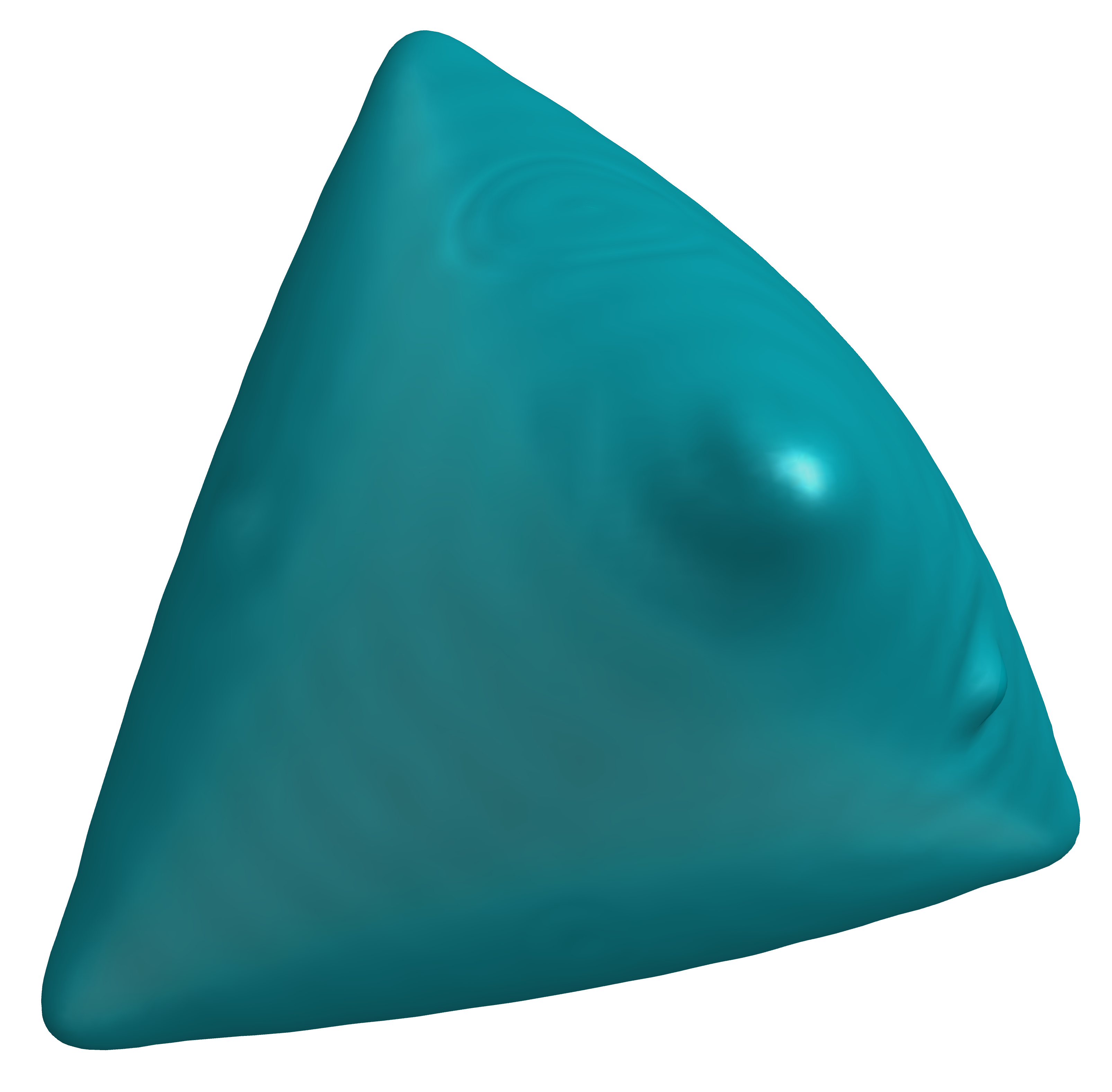}
        \includegraphics[height=3.8cm]{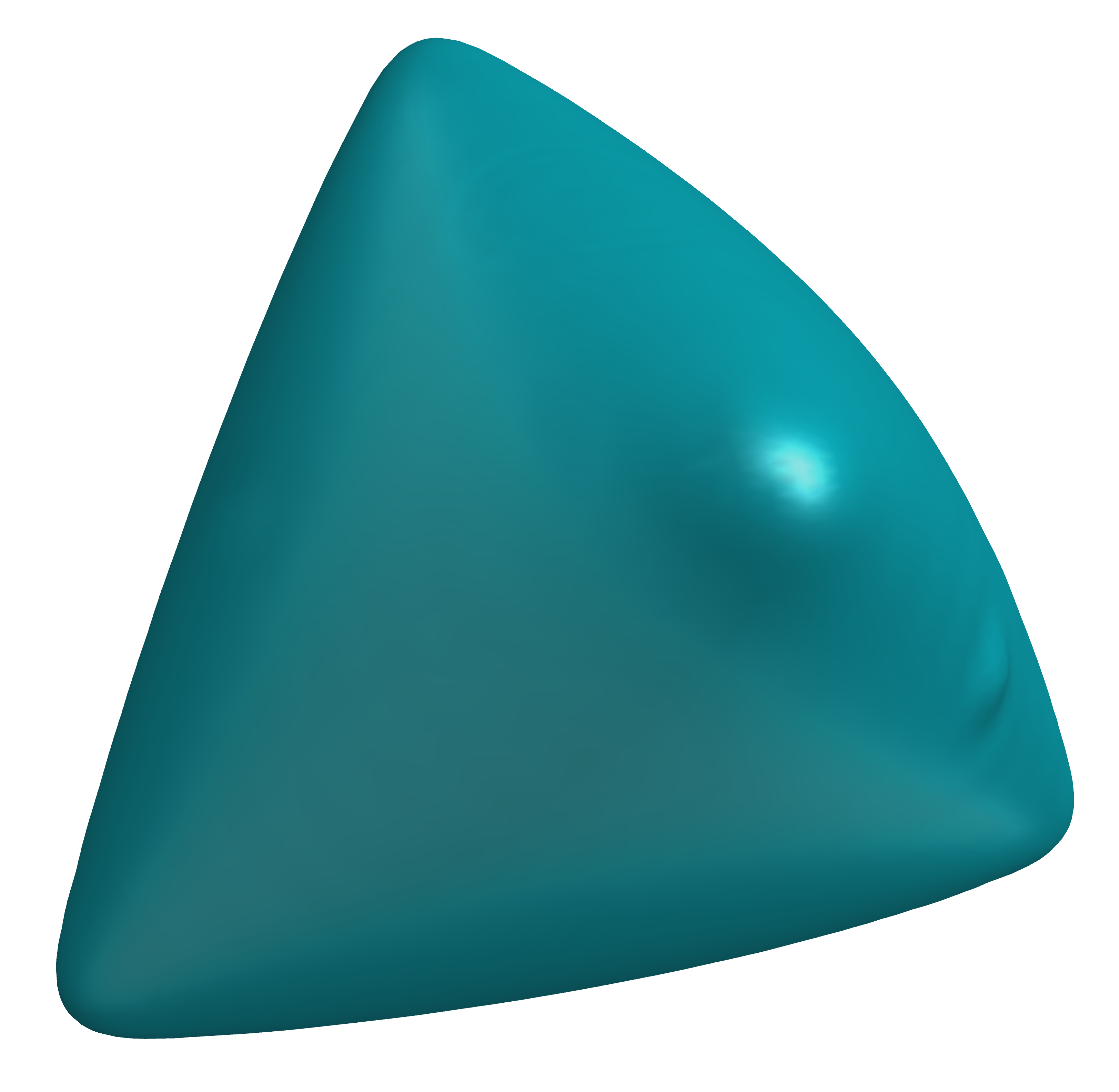}
        \includegraphics[height=3.8cm]{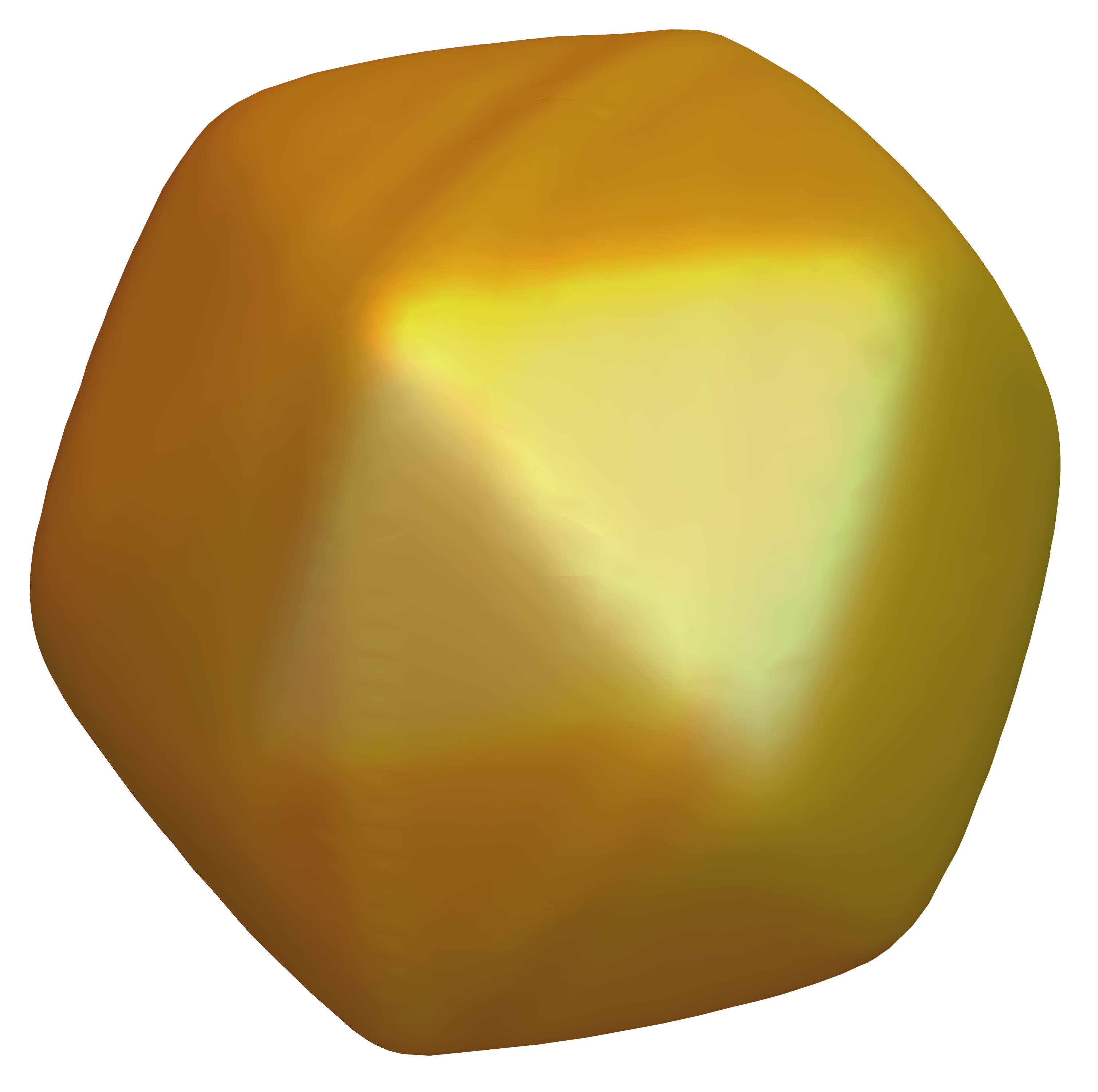}
    \fi
        \caption{From left to right: the violet elliptope $\mathcal{E}$ from Example \ref{ex:elliptope} and its green slice with largest volume; the teal polynomial approximation of $L$, via Algorithm \ref{alg:approx_fourier}; the teal \emph{mollified} polynomial approximation of $\mathcal{E}$, via Algorithm \ref{alg:approx_mollified}; the yellow polynomial approximation of the intersection body of $\mathcal{E}$, via Algorithm \ref{alg:IB}. All approximating polynomials have degree $30$.}
        \label{fig:ex_elliptope}
    \end{figure}
\begin{table}[!ht]
    \centering
    \begin{tabular}{c||c|c|c|c}
        $L$ & numerical max & numerical direction & max & a direction \\
        \hline
        cube & $5.4215$ & $(-0.7070,-0.7071, 0)$ & $4\sqrt{2}$ & $(-\frac{1}{\sqrt{2}},-\frac{1}{\sqrt{2}},0)$ \\ %[1ex]
        cylinder & $4.1184$ & $(0.6414, 0.2595, -0.7219)$ & $\pi\sqrt{2}$ & $(\frac{\cos{t}}{\sqrt{2}}, \frac{\sin{t}}{\sqrt{2}}, -\frac{1}{\sqrt{2}})$ \\ %[1ex]
        dented tin can & $4.1713$ & $(0.6314,0.3021,-0.7141)$ & $\pi\sqrt{2}$ & $(\frac{\cos{t}}{\sqrt{2}}, \frac{\sin{t}}{\sqrt{2}}, -\frac{1}{\sqrt{2}})$ \\ %[1ex]
        elliptope & $3.7260$ & $(0.7089, -0.0026, -0.7052)$ & $\frac{8\sqrt{2}}{3}$ & $(\frac{1}{\sqrt{2}},0,-\frac{1}{\sqrt{2}})$ \\
    \end{tabular}
    \caption{Largest volume slice of $L$. The first column contains the considered convex body; the second and third columns contain the value for the largest volume and the direction orthogonal to the corresponding slice, obtained from the numerical computation (with four digit precision); the fourth and fifth columns contain the largest volume of the slices of $L$ and one of the directions orthogonal to one of them, respectively.}
    \label{tab:volume_slice}
\end{table}

\begin{remark}
    The interest in the largest volume slice of a convex body is motivated by the recent resolution of Bourgain's slicing conjecture \cites{Guan24:Bourgain,KlaLeh24:Bourgain}. Indeed, the volume of the largest volume slice of any convex body of volume $1$ is a lower bound for the (unknown) universal constant in Bourgain's problem. Further developing a solid and efficient computational frameworks for the largest volume slice computation also in higher dimension (in fact, the first relevant dimension is $5$) could provide us with nontrivial lower bounds for the universal constant.
\end{remark}

\subsection{Computing the width of dual bodies}\label{subsec:width}
The width of a set is the minimum distance $w$ such that it is possible to fit the body between two parallel hyperplanes that are at distance $w$ apart. Recall that the gauge function of a convex set $L$ coincides with the \emph{support function} of $L^\circ$, which is by definition
\begin{equation}
    h_{L^\circ}(x):=\max\{\langle x,y \rangle: y\in L^\circ\}.
\end{equation}
It follows that the width $w(L^{\circ})$ of the polar body $L^{\circ}$ can be computed as
\[
w(L^{\circ}) = \max_{x\in S^{n-1}}\left(\gamma_L(x)+\gamma_L(-x)\right).
\]
By Theorem~\ref{polySOS}, if $L$ is a convex polygauge body then the computation of the width of $L^{\circ}$ reduces to the optimization of an even regular (polynomial) function on the sphere, a nontrivial problem for which several approaches are available. The following examples will show that this approach can be carried out in practice in some cases. 

A possible procedure for approximating the width of a convex body $L$ would be to approximate $L^\circ$ with a polygauge body $(L^\circ)_d$ whose gauge function is a polynomial of degree $d$, obtained from Algorithm \ref{alg:approx_mollified} applied to $f_{L^\circ}$, and maximize this polynomial. 
However, in order to make use the computations made for the examples in Section \ref{sec:examples} and highlight how many information can be deduced from a single polynomial approximation, we use the following approach. We compute the polyradial approximation $(L^\circ)_d$ of $L$, and then numerically maximize the rational function $\widetilde{(\rho_{L^\circ})}_{d}^{-1}$ to get the width of $L$.

\begin{example}
    We continue the examples in Section \ref{sec:examples} using the approach just described. The starbodies in Examples \ref{ex:cube}, \ref{ex:cylinder}, and \ref{ex:elliptope} are convex, and we can therefore think of each of them as the dual body $L^\circ$ to a convex body $L$, where $L$ is respectively 
    \begin{itemize}
        \item the octahedron $L=\conv \{(\pm 1, 0, 0), (0,\pm1, 0), (0,0,\pm 1)\}$, dual to the cube;
        \item the spinning top $L = \{ (x,y,z)\in [-1,1]^3 \,\colon\, x^2+y^2\leq (1-z)^2, x^2+y^2\leq (1+z)^2\}$, dual to the cylinder;
        \item the smoothed tetrahedron $L = \conv \{  (x,y,z)\in [-1,1]^3 \,\colon\, x^2 y^2 + y^2 z^2 + x^2 z^2 - 2 x y z=0\}$, dual to the elliptope.
    \end{itemize}
    These convex bodies are displayed in Figure \ref{fig:width}, from left to right. 
    \begin{figure}[h]
        \centering
        \includegraphics[height=4.1cm]{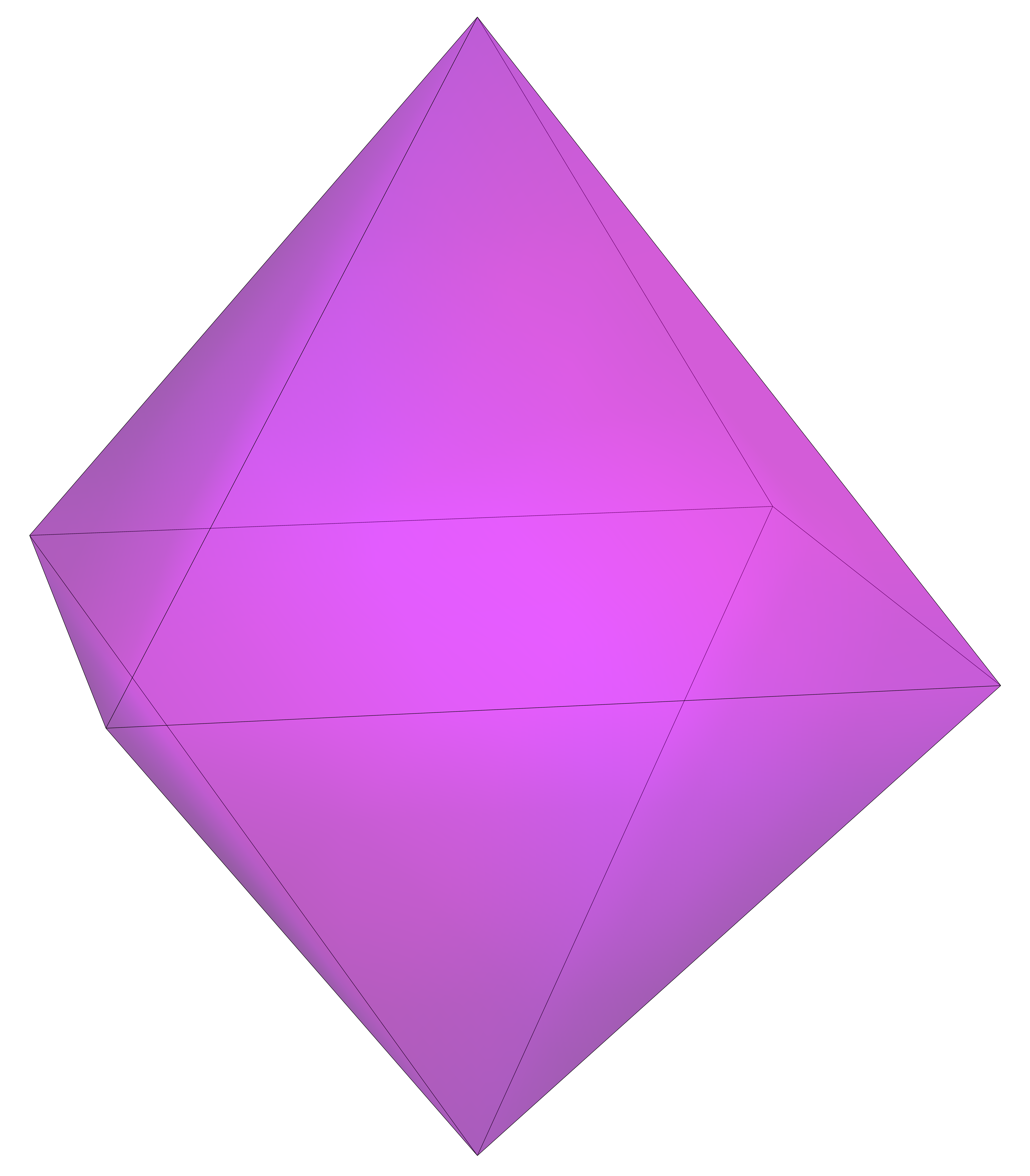}
        \quad
        \includegraphics[height=4cm]{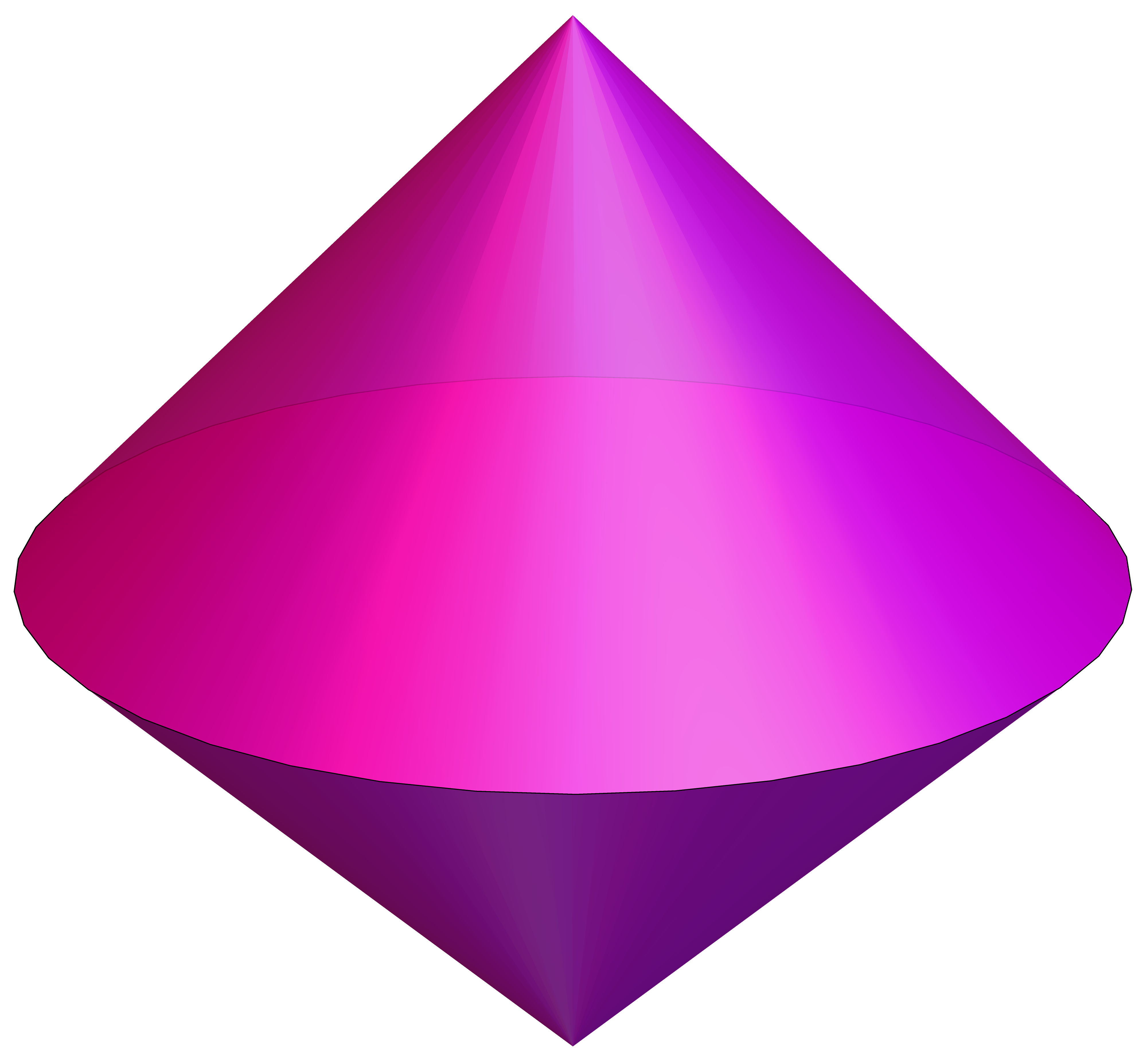}
        \quad
        \includegraphics[height=3.8cm]{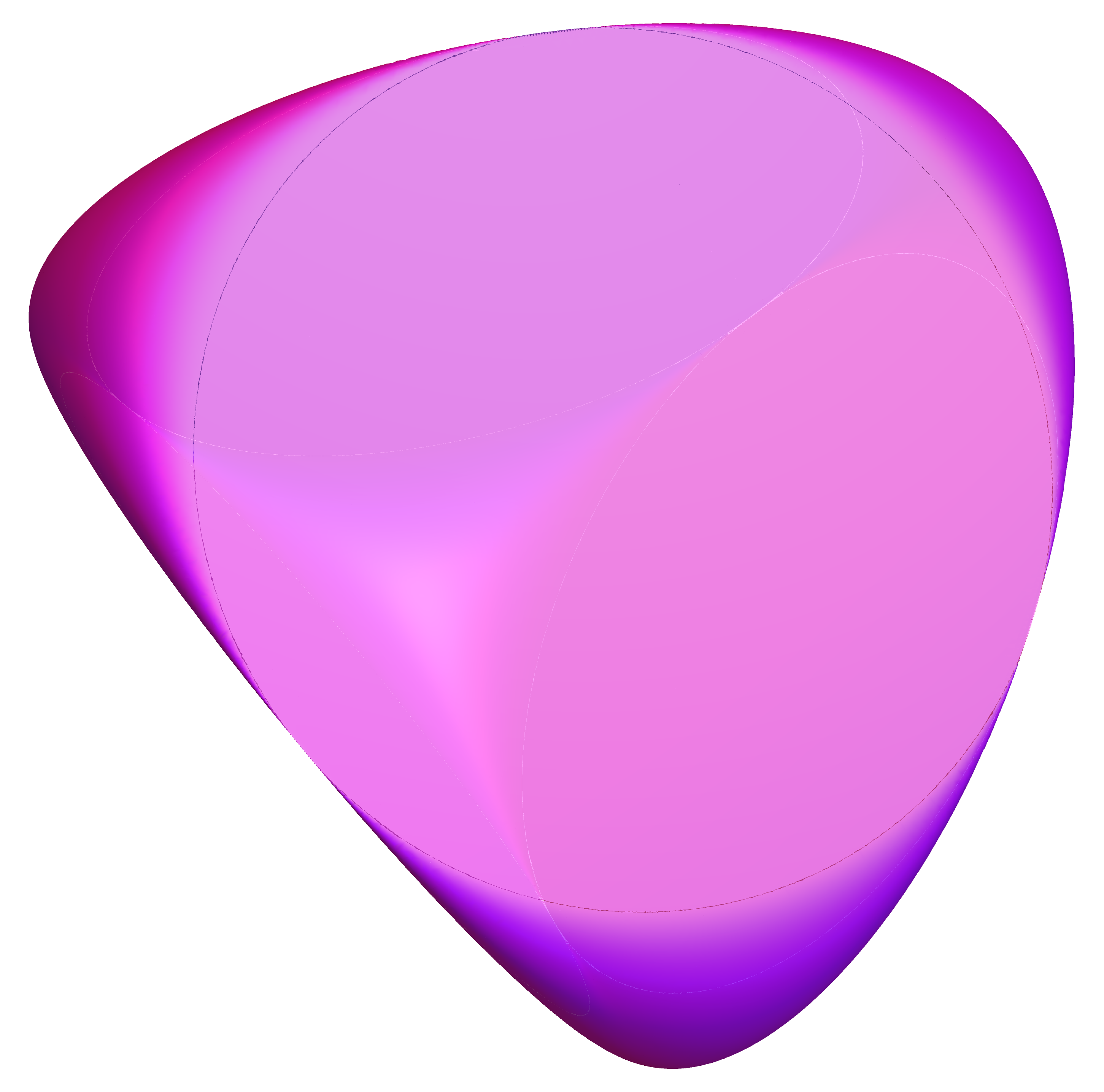}
        \caption{From left to right: dual bodies to cube, cylinder, and elliptope.}
        \label{fig:width}
    \end{figure}
    Since, for all these choices of $L$, we have access to $\widetilde{(\rho_{L^\circ})}_{30}$, the degree $30$ mollified polynomial approximation of $\rho_{L^\circ}$ which was computed in Examples \ref{ex:cube}, \ref{ex:cylinder}, and \ref{ex:elliptope}, we use the following approach.
    We consider
    \[
    \widetilde{h}_L = \frac{1}{\widetilde{(\rho_{L^\circ})}_{30}},
    \]
    which is a rational function that approximates $h_L$, or equivalently $\gamma_{L^\circ}$, and we use numerical methods to find the maximum of the sum $\widetilde{h}_L(x)+\widetilde{h}_L(-x)$ for $x\in S^2$. Concretely, this is done in \texttt{Mathematica} using the command \texttt{NMaximize}, and heuristic methods such as the Nelder-Mead simplex method. The output of such computation is summarized in Table \ref{tab:width}.
    \begin{table}[!ht]
        \centering
        \begin{tabular}{c||c|c|c|c}
            $L$ & numerical width & numerical direction & width & a direction \\
            \hline
            octahedron & $1.9980$ & $(0,0,-1)$ & $2$ & $(0,0,-1)$ \\
            double cone & $1.9957$ & $(0.0674, -0.07364, -0.9950)$ & $2$ & $(0,0,-1)$ \\
            smoothed tetrahedron & $2.0004$ & $(-0.0081, 0, -1)$ & $2$ & $(0,0,-1)$ \\
        \end{tabular}
        \caption{Computation of the width of $L$. The first column contains the convex body we compute the width of; the second and third columns contain the value for the width and the direction in which that is achieved, obtained from the numerical computation (with four digit precision); the fourth and fifth columns contain the exact width of $L$ and one of the directions in which it is achieved, respectively.}
        \label{tab:width}
    \end{table}
\end{example}

\bibliographystyle{alpha}
\bibliography{biblio}

\vfill
\small{

\noindent \textsc{Chiara Meroni} \\
\textsc{ETH Institute for Theoretical Studies, Z\"urich, Switzerland} \\
\url{chiara.meroni@eth-its.ethz.ch} \\

\noindent \textsc{Jared Miller}\\
\textsc{University of Stuttgart, Stuttgart, Germany} \\
\url{jared.miller@imng.uni-stuttgart.de} \\

\noindent \textsc{Mauricio Velasco} \\
\textsc{Centro de Matemáticas (CMAT), Universidad de la Rep\'ublica, Montevideo, Uruguay} \\
\url{mvelasco@cmat.edu.uy} \\
}

\end{document}